\documentclass[11pt]{article}
\usepackage{amsmath,amsthm,amsfonts,amssymb,bm,wasysym}
\usepackage{epsfig}
\usepackage[usenames]{color}
\usepackage{verbatim}
\usepackage{hyperref}
\usepackage{multicol}
\usepackage{comment}
\usepackage{float}
\usepackage{graphicx}

\parskip \medskipamount
\newtheorem{theorem}{Theorem}[section]
\newtheorem{definition}[theorem]{Definition}

\numberwithin{equation}{section}
\newtheorem{lemma}[theorem]{Lemma}

\newtheorem{corollary}[theorem]{Corollary}

\newtheorem{remark}[theorem]{Remark}

\newtheorem{claim}[theorem]{Claim}

\numberwithin{equation}{section}

\def\N{\mathbb{N}}
\def\Z{\mathbb{Z}}

\def\R{\mathbb{R}}

\def\EE{\mathcal{E}}

\def\LL{\mathcal{L}}

\renewcommand{\phi}{\varphi}
\renewcommand{\epsilon}{\varepsilon}

\allowdisplaybreaks

\newcommand{\1}{{\text{\Large $\mathfrak 1$}}}

\renewcommand{\emptyset}{\varnothing}

\newcommand{\til}{\widetilde}

\newcommand{\tcov}{\tau_{\mathrm{cov}}}

\newcommand{\tmix}{t_{\mathrm{mix}}}

\newcommand{\tunif}{t_{\mathrm{unif}}}

\newcommand{\pr}[1]{\mathbb{P}\!\left(#1\right)}
\newcommand{\E}[1]{\mathbb{E}\!\left[#1\right]}
\newcommand{\estart}[2]{\mathbb{E}_{#2}\!\left[#1\right]}
\newcommand{\prstart}[2]{\mathbb{P}_{#2}\!\left(#1\right)}

\newcommand{\econd}[2]{\mathbb{E}\!\left[#1\;\middle\vert\;#2\right]}

\newcommand{\tn}{|\kern-.1em|\kern-0.1em|}

\newcommand{\vr}[1]{\mathrm{Var}\left(#1\right)}

\newcommand{\vrstart}[2]{\mathrm{Var}_{#2}\left(#1\right)}

\newcommand\be{\begin{equation}}
\newcommand\ee{\end{equation}}

\newcommand{\tv}[1]{\left\|#1\right\|_{\rm{TV}}}

\newcommand{\id}[1]{d_\textrm{in}(#1)}
\newcommand{\outd}[1]{d_\textrm{out}(#1)}

\begin{document}
\title{\bf Sensitivity of mixing times in Eulerian digraphs}

\author{Lucas Boczkowski \thanks{Universit\'e Paris Diderot, Paris, France; lucasboczko@gmail.com} \and
Yuval Peres\thanks{Microsoft Research, Redmond, Washington, USA; peres@microsoft.com} \and Perla Sousi\thanks{University of Cambridge, Cambridge, UK;   p.sousi@statslab.cam.ac.uk} 
}
\date{}
\maketitle
\begin{abstract}
Let $X$ be a lazy random walk on a graph $G$. If $G$ is undirected, then the mixing time is upper bounded by the maximum hitting time of the graph. This fails for directed chains, as the biased random walk on the cycle $\Z_n$ shows. However, we establish that for Eulerian digraphs, the mixing time is $O(mn)$, where $m$ is the number of edges and $n$ is the number of vertices. In the reversible case, the mixing time is robust to the change of the laziness parameter. Surprisingly, in the directed setting the mixing time can be sensitive to such changes. We also study exploration and cover times for random walks on Eulerian digraphs and prove universal upper bounds in analogy to the undirected case.
\newline
MSC 2010 \emph{subject classifications.} Primary 60J10.
\end{abstract}

\section{Introduction}\label{sec:intro}

Random walks on graphs have been thoroughly studied, 
but many results are only known in the undirected setting, 
where spectral methods and the connection with electrical networks is available. As noted in Aldous and Fill~\cite{AF}, many properties of random walk on undirected graphs extend to Eulerian digraphs. In this paper we show this holds true for basic bounds on mixing and exploration times, but fails for robustness of mixing times under modification of laziness parameters.

\begin{definition}\rm{
Let $G=(V,E)$ be a directed graph, where $V$ is the set of vertices and $E$ the set of directed edges. We write $\outd{x}$ and $\id{x}$ for the outdegree and indegree of the vertex~$x$ respectively. The graph $G$ is {\textbf{connected}} if for all vertices $x$ and $y$ there is a path from $x$ to $y$ ignoring directions.
It is called~$d$-{\textbf{regular}} if $\outd{x}=\id{x}=d$ for all $x\in V(G)$. Finally it is called \textbf{Eulerian} if $\outd{x}=\id{x}$ for all $x\in V(G)$. 
}
\end{definition}

Let $X$ be a Markov chain on the finite state space $E$ with transition matrix $P$ and stationary distribution $\pi$. We write $P^t(x,y) = \prstart{X_t=y}{x}$ for all $x,y\in E$. 
The chain~$X$ is called lazy if~$P(x,x)\geq 1/2$ for all~$x\in E$.

Let $G= (V,E)$ be a directed graph. Then a lazy simple random on $G$ is a Markov chain $X$ with transition probabilities given by
\begin{equation*}
P(v,v) = \frac{1}{2} \quad \text{and}\quad  P(v,w)= \frac{1}{2\outd{v}}\textsf{1}((v,w) \in {E}) \quad \forall\, v,w\in V.
\end{equation*}
For convenience, when we consider Eulerian graphs, we drop the subscript ``out''. For all $\epsilon>0$ we define the $\epsilon$-total variation mixing time of the transition probability matrix $P$ via
\[
\tmix(\epsilon) = \min\{t\geq 0: \max_{x} \tv{P^t(x,\cdot) - \pi}\leq \epsilon\}
\]
and the $\epsilon-\LL_\infty$ mixing via
\[
\tunif(\epsilon) = \min\left\{t\geq 0: \max_{x,y}\left|\frac{P^t(x,y)}{\pi(y)} -1\right|\leq \epsilon\right\}.
\]
Following the convention, we write $\tmix=\tmix(1/4)$ and $\tunif=\tunif(1/4)$. 

If $X$ is reversible and lazy, then it is known (\cite[Chapter~10]{LevPerWil} or \cite{RussShayan}) that 
\[
\tunif \leq 4\max_{x,y\in E} \estart{\tau_y}{x}+1,
\]
where $\tau_y=\inf\{t\geq 0: X_t=y\}$ for all $y\in E$. 

However, if $X$ is not reversible, then this upper bound fails. Indeed, if $X$ is a lazy biased walk on~$\Z_n$ with $P(i,i)=1/2$ and $P(i,i+1)=1/3=1/2-P(i,i-1)$, where all the operations are done modulo $n$, then~$\tunif$ is of order $n^2$, while $\max_{x,y}\estart{\tau_y}{x}$ is of order $n$. This chain can be realised as a random walk on an Eulerian multi-graph.

It is well-known that if $X$ is a lazy simple random walk on an undirected connected graph $G$, then there is a positive constant $c$ so that
\[
\max_{x,y}\estart{\tau_y}{x} \leq c\cdot mn,
\]
where $m$ is the number of edges and $n$ the number of vertices of $G$. Therefore, when $G$ is undirected then 
\[
\tunif\leq c' \cdot mn,
\]
for a positive constant $c'$. It turns out that the same upper bound is true in the Eulerian directed case as the following theorem shows. 

Recall that an Eulerian digraph is strongly connected if and only if it is connected. Therefore, a random walk on a connected Eulerian graph is always irreducible.

\begin{theorem}\label{mixingthe}
There exists a positive constant $c_1$ so that for all lazy walks on connected Eulerian digraphs on~$n$ vertices and~$m$ directed edges the following bound holds
\[
\tunif\leq c_1 \cdot mn.
\]
More precisely, there exists a constant $c_2$ so that for all $a$
\[
\tunif(a)\leq \frac{c_2}{a}\left((mn)\wedge \frac{m^2}{a}\right).
\]
Moreover, if $G$ is regular, then there exists a constant $c_3$ so that for all $a$
\[
\tunif(a)\leq c_3 \cdot \frac{n^2}{a^2}.
\]
\end{theorem}

\begin{remark}\rm{
We note that for $a<1$, the bounds can be improved using sub-multiplicativity to
\[
\tunif(a)\leq c_1 mn \cdot  \log(1/a)
\]
for general Eulerian digraphs, and 
\[
\tunif(a) \leq c_3 n^2 \cdot \log(1/a)
\]
in the regular case.
}	
\end{remark}

%

Montenegro and Tetali in~\cite[Example~6.4]{MontTetali} prove a bound of $O(m^2)$ on the total variation and~$\LL^2$~mixing time for random walk on connected Eulerian graphs under an expansion hypothesis. Our bound $O(mn)$ of Theorem~\ref{mixingthe} improves on their bound considerably in the lazy case especially when the degrees are high.

We prove Theorem~\ref{mixingthe} in Section~\ref{sec:mixing}, where we also derive an upper bound on hitting times of moving targets.

We now turn to a property of undirected graphs that does not transfer to the directed Eulerian setting. K.\ Burdzy asked whether the mixing time of a chain is robust to modifications of laziness parameters (explained formally below). A positive answer for reversible chains was given in~\cite{PSMixhit}:

\begin{theorem}{\rm{\cite[Corollary~9.5]{PSMixhit}}}
For all $c_1,c_2\in (0,1)$, there exist two positive constants~$c$ and $c'$ so that the following holds. Let~$P$ be an irreducible reversible transition matrix on the state space~$E$ and let $\widetilde{P}=(P+I)/2$ have mixing time $\tmix$. Suppose that $(a(x,x))_{x\in E}$ satisfy $c_1\leq a(x,x) \leq c_2$ for all~$x\in E$. Let $Q$ be the transition matrix of the Markov chain with transitions: when at $x$ it stays at $x$ with probability $a(x,x)$. Otherwise, with probability $1-a(x,x)$ it jumps to state~$y\in E$ with probability $P(x,y)$. We then have
\[
c\tmix\leq \tmix(Q) \leq c' \tmix,
\]
where $\tmix(Q)$ is the mixing time of the transition matrix $Q$.
\end{theorem}

Surprisingly, in the directed case the mixing time is sensitive to the change of the laziness parameter as the following theorem demonstrates. 

\begin{theorem}\label{thm:robustness}
Let $X$ be a random walk on the graph $G$ of Figure~\ref{fig:twocycle}, which is obtained by gluing two cycles $\Z_n$ at $0$.  At sites on the left cycle which are at distance in $[n/4,3n/4]$ from~$0$, the random walk stays in place with probability $\alpha=2/(\sqrt{5}+1)$ and everywhere else with probability $1/2$. With the remaining probability it jumps to a neighbour according to a biased walk, i.e.\ with probability~$2/3$ clockwise and $1/3$ counter-clockwise, except when at~$0$ it additionally chooses with equal probability on which cycle to start walking. Then there exist two positive constants $c_1$ and $c_2$ so that for all $n$ 
\[
c_1n^{3/2}\leq \tmix\leq  \tunif \leq c_2 n^{3/2}.
\]
\end{theorem}

This contrasts with the case $\alpha=1/2$, where $\tmix, \tunif \asymp n^2$.

\begin{figure}[h!]
\begin{center}
\includegraphics[scale=0.7]{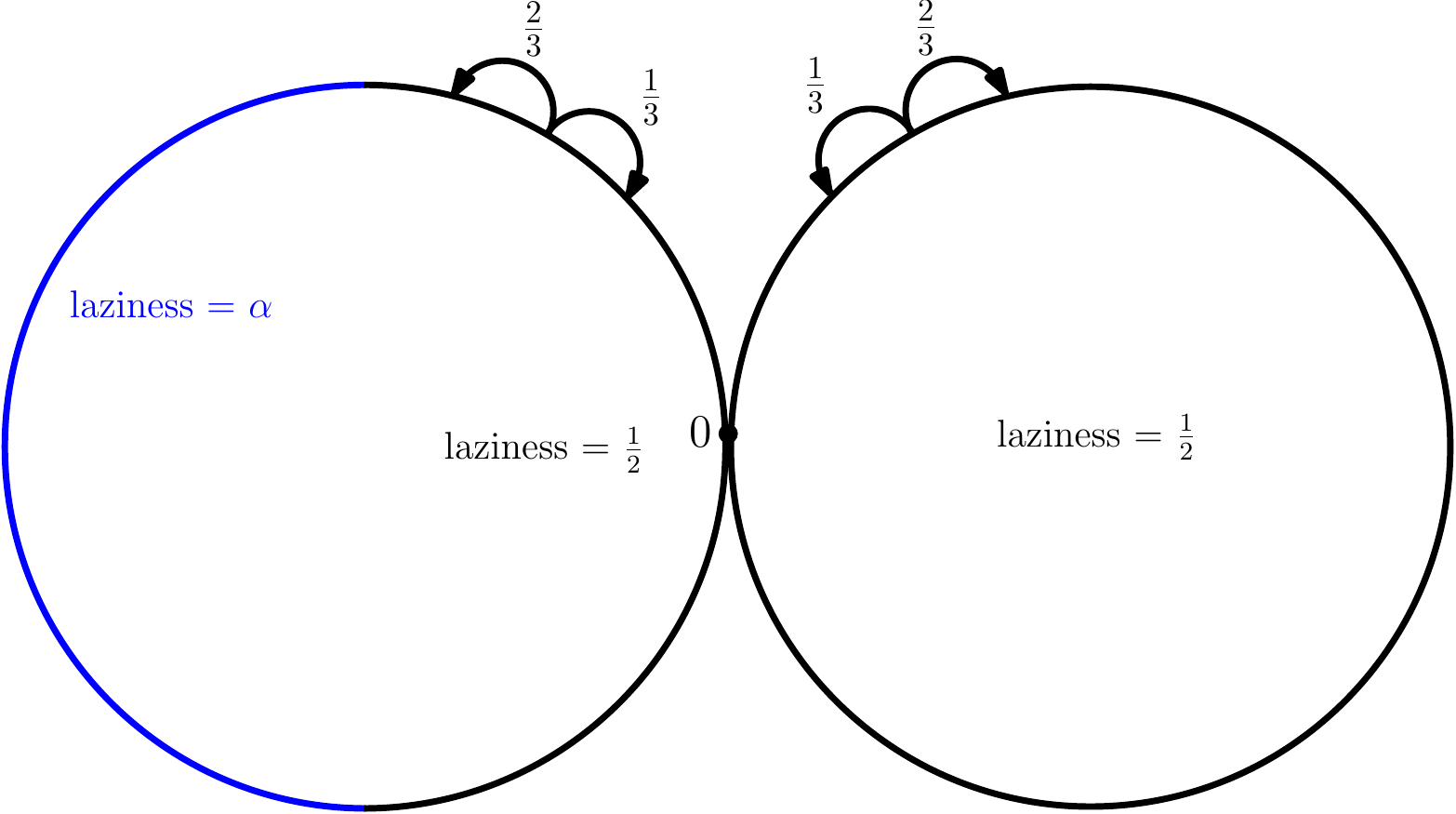}
\caption{Biased random walk on the cycles with different laziness parameters}\label{fig:twocycle}
\end{center}
\end{figure}

We prove Theorem~\ref{thm:robustness} in Section~\ref{sec:sensitivity}.
We now turn to properties of random walks on undirected graphs that also extend to the directed case. Before stating the results we introduce some notation.

\begin{definition}\rm{
Let $G$ be a graph and $X$ a simple random walk on $G$.
For $k\in \N$ the $k$-\textbf{exploration time}, $T_k$, is the first time that $X$ has visited $k$ distinct vertices of $G$. When $k=n$, then $T_n=\tcov$ is called the \textbf{cover time} of the graph.
}
\end{definition}

When the graph is undirected, Barnes and Feige~\cite{BF} showed that $\estart{T_k}{v}\leq Ck^3$ for a universal constant $C$ and all starting vertices $v$.
In the case of regular graphs, the bound $\estart{T_k}{v}\leq Ck^2$ follows from standard results (see for instance \cite[Chapter~6]{AF}). 
The following results extend these bounds to the directed Eulerian setting. We prove them in Section~\ref{sec:exploration}.

%

\begin{theorem}\label{sec}
There exists a positive constant $c_1$ so that if~$G$ is a regular connected Eulerian digraph on $n$ vertices, then for all starting vertices~$v$ and all $k\leq n$ we have
\begin{equation*}
\estart{T_k}{v} \leq c_1k^2.
\end{equation*}
\end{theorem}

\begin{theorem}\label{primo}
There exists a positive constant $c_2$ so that if~$G = (V, E)$ is a connected Eulerian digraph on $n$ vertices, then for all starting vertices $v$ and all $k\leq n$ we have
\begin{equation*}
\estart{T_k}{v} \leq c_2 k^3.
\end{equation*}
\end{theorem}


\textbf{Notation}.
For functions $f,g$ we will write $f(n) \lesssim g(n)$ if there exists a constant $c > 0$ such that $f(n) \leq c g(n)$ for all $n$.  We write $f(n) \gtrsim g(n)$ if $g(n) \lesssim f(n)$.  Finally, we write $f(n) \asymp g(n)$ if both $f(n) \lesssim g(n)$ and $f(n) \gtrsim g(n)$.

\section{Bounds on mixing times}\label{sec:mixing}

In this section we prove Theorem~\ref{mixingthe}, which gives an upper bound for the $\LL_\infty$ mixing time. 
There are several well developed tools for analysing mixing times, most of which are not effective in the directed case. Our bounds rely on the spectral profile technique introduced in \cite{GMT}. We stress that no other technique we are aware of yields results of this precision.

We start by recalling the spectral profile technique of~\cite{GMT}.

Let $X$ be a Markov chain on $V$ with transition matrix $P$ and stationary distribution $\pi$. The Dirichlet energy of a function $f:S\to \R$ is defined via
\[
\EE_P(f,f): = \frac{1}{2}\sum_{v,w \in V}  \left(f(v) - f(w) \right)^2\pi(v) P(v,w).
\]
For $f:S\to \R$ and $g:S\to \R$ we define their inner product with respect to $\pi$ via
\[
\langle f,g\rangle_\pi = \sum_x f(x) g(x)\pi(x).
\]
Note that for a reversible chain we have 
\begin{align}\label{eq:dirichletrev}
	\EE_P(f,f) = \langle f, (I-P)f\rangle.
\end{align}
For a subset $\emptyset\neq S\subseteq V$ we define $\lambda(S)$ via
\begin{align*}
\lambda(S) = \inf_{f \geq 0,\, \text{supp}(f) \subseteq S} \frac{\mathcal{E}_P(f,f)}{\text{Var}_{\pi}(f)},
\end{align*}
where $\text{supp}(f)=\{x: f(x)\neq 0\}$ and ${\text{Var}_{\pi}(f)} = \estart{(f-\estart{f}{\pi})^2}{\pi}$. 

The spectral profile $\Lambda : \left[\pi_{*}, \infty \right) \rightarrow \mathbb{R}$ is defined by
\begin{align*}
\Lambda(r) = \inf_{\pi_* \leq \pi(S) \leq r} \lambda(S),
\end{align*}
where $\pi_*=\min_x\pi(x)$.

We can now state the main result of~\cite{GMT} that we will use.
\begin{theorem}[\cite{GMT}]\label{gmt}
Let $P$ be an irreducible transition matrix with stationary distribution $\pi$ satisfying $P(x,x)\geq \delta>0$ for all $x$. Then for all $a>0$ we have
\begin{align*}
\tunif(a) \leq 2\left\lceil \int_{4\pi_*}^{4/a} \frac{dr}{\delta r \Lambda(r)} \right\rceil.
\end{align*}
\end{theorem}

The following lemma is standard. We will use it in the proof of Theorem~\ref{mixingthe}, so we include the proof here for completeness. 

\begin{lemma}\label{lem:var}
For all $r\leq 1/2$ the spectral profile satisfies
\begin{align}\label{sand}
\inf_{f \geq 0, \, \pi({\rm{supp}}(f)) \leq r} \frac{\mathcal{E}_P(f,f)}{\estart{f^2}{\pi}} \leq \Lambda(r) \leq \inf_{f \geq 0, \,\pi({\rm{supp}}(f)) \leq r} \frac{2\mathcal{E}_P(f,f)}{\estart{f^2}{\pi}}
\end{align}
\end{lemma}

\begin{proof}[\bf Proof]

For a function $f$ and a set $A$ the Cauchy-Schwartz inequality gives 
\[
\left(\E{f\1(A)}\right)^2 \leq \E{f^2\1(A)} \pi(A).
\]
Taking $A={\rm{supp}}(f)$ we get 
\[
\vrstart{f}{\pi} \geq (1-\pi(A)) \estart{f^2}{\pi},
\]
and this proves the lemma.
\end{proof}

We are now ready to give the proof of Theorem~\ref{mixingthe}.

\begin{proof}[\bf Proof of Theorem~\ref{mixingthe}]

We will use Theorem~\ref{gmt}, and hence we need to find a lower bound on the spectral profile~$\Lambda$. Consider a set $S$ with $\pi(S)\leq r$.

Let $\widehat{P}$ the matrix of the reversed chain, i.e.\ the matrix defined by
\begin{equation*}
\pi(v) \widehat{P}(v,u) = \pi(u) P(u,v) \quad \text{for all } u,v.
\end{equation*}
Since the graph is Eulerian, taking the reversal amounts to reversing the orientation of the edges. Let also
\[
Q=\frac{P+\widehat{P}}{2}.
\]
Then the chain with matrix $Q$ is a reversible chain and can be seen as a walk on a weighted graph. An edge appearing in the original graph in both directions has weight $2$, other edges have weight~$1$. 

Using Lemma~\ref{lem:var} shows that we can replace the variance in the definition of $\Lambda$ by an $\LL_2$ norm. 

It is immediate to check by rearranging the terms in the definition of the Dirichlet energy that 
\begin{align}\label{eq:eff}
\EE_P(f,f) = \EE_{\widehat{P}}(f,f) = \EE_{Q}(f,f).
\end{align}
Because of this, we can work with the reversible chain $Q=(P+\widehat{P})/2$. Therefore, using~\eqref{eq:dirichletrev} the ratio we are interested in lower bounding becomes
\begin{align*}
\lambda(S)=\inf_{f\geq 0, \,\text{supp}(f) \subseteq S} \frac{\langle f, (I-Q)f \rangle_{\pi}}{\estart{f^2}{\pi}} =: 1 - \rho.
\end{align*}
A function $f$ minimising the above expression is a left eigenfunction of the restriction~$\til{Q}$ of $Q$ to $S$, for some eigenvalue $\rho$. The matrix $\widetilde{Q}$ is an $|S|\times|S|$ strictly sub-stochastic matrix which implies that~$\rho < 1$ and $f$ is strictly positive by the Perron Frobenius theorem. We now set $\til{f}(x) = \pi(x) f(x)$ for all $x$. Then the reversibility of $Q$ yields that~$\til{f}$ is a right eigenvector of $\til{Q}$.

Define $\nu = \frac{\til{f}}{\|\til{f}\|_1}$ the $\LL_1$-renormalisation of $\til{f}$. By definition
\[
\nu Q =  \nu\widetilde{Q}+ \xi = \rho \nu + \xi,
\]
for some vector $\xi$ supported on~$S^c$. This shows that starting with initial distribution $\nu$ on $S$, the exit time from $S$ is a geometric variable with mean $(1-\rho)^{-1}$. Therefore, we get
\[
\lambda(S)= \frac{1}{\estart{\tau(S^c)}{\nu}} \geq \frac{1}{\displaystyle\max_{v\in S}\estart{\tau(S^c)}{v}},
\]
where $\mathbb{E}_{v}$ is w.r.t the chain $Q$, which as explained above is a reversible chain. Hence taking the infimum over all $S$ with $\pi(S)\leq r$ gives 
\[
\Lambda(r) \geq \frac{1}{\displaystyle\max_{S: \pi(S)\leq r}\max_{v\in S}\estart{\tau(S^c)}{v}}.
\]
It now remains to upper bound~$\estart{\tau(S^c)}{v}$, where $v\in S$, for the reversible chain with matrix~$Q$. The distance from $v$ to $S^c$ is at most $|S|$. Therefore using the commute time identity for reversible chains~\cite[Proposition~10.6]{LevPerWil} we obtain
\begin{align*}
\estart{\tau(S^c)}{v} \leq |S|\cdot rm \leq (rm)^2\wedge (rmn).
\end{align*}
This now gives a lower bound on the spectral profile
\[
\Lambda(r) \geq  \left((rm)^2\wedge(rmn)\right)^{-1}.
\]
Plugging this in the integral of Theorem~\ref{gmt} gives for a positive constant $c_1$ that
\begin{align*}
\tunif(a) \leq \frac{c_1}{a}\left((mn)\wedge \frac{m^2}{a}\right)
\end{align*}
and this concludes the proof when $G$ is an Eulerian digraph. 

When $G$ is a regular digraph, then the chain $Q$ corresponds to a simple random walk on a regular undirected graph. If $S$ is a set with $|S|\leq rn$, then we can apply~\cite[Proposition~6.16]{AF} to obtain 
\[
\estart{\tau(S^c)}{v} \leq 4|S|^2 \leq 4(rn)^2.
\]
Similarly as above, Theorem~\ref{gmt} now gives
\[
\tunif(a) \lesssim \frac{n^2}{a^2}
\]
and this now finishes the proof.
\end{proof}

\subsection{Short-time bounds and moving target.}

In this section we give an application of Theorem~\ref{mixingthe} to short time estimates for transition probabilities and hitting times for moving targets. 

Let $u=(u_s)_{s\geq 0}$ be a deterministic trajectory in the graph, i.e.\ $u:\N \to V(G)$. We define the first collision time $\tau_{\rm{col}}$ via
\[
\tau_{\rm{col}} = \inf\{ t\geq 0: X_t=u_t\},
\] 
where $X$ is a simple random walk on $G$.

\begin{corollary}\label{movingtarget}
There exists a positive constant $c$ so that the following holds. Let $G$ be an Eulerian digraph on $n$ vertices and $m$ edges and $X$ a lazy simple random walk on it. 
Let $(u_t)_{t\geq 0}$  with $u_t\in V(G)$ for all $t$ be a moving target on $G$. Then 
\[
\E{\tau_{\rm{col}}} \leq c  \cdot (mn)\cdot (1+\log(m/n)).
\]
Moreover, if $G$ is regular, then 
\[
\E{\tau_{\rm{col}}} \leq c \cdot n^2.
\]
\end{corollary}

The following lemma will be crucial in the proof of the corollary above.

\begin{lemma}\label{lem:shorttime}
There exists a positive constant $C$ so that the following holds. Let $G$ be an Eulerian digraph on $n$ vertices and $m$ edges. If $P$ is the transition matrix of a lazy simple random walk on~$G$, then we have
\begin{align}\label{bdd} 
\left|\frac{P^t(u,v)}{\pi(v)} -1 \right| \leq \begin{cases} C\frac{m}{\sqrt{t}} &\mbox{if } t \leq n^2, \\ 
C\frac{mn}{t} & \mbox{if } t \geq n^2.\end{cases} 
\end{align}
Moreover, if $G$ is regular, then for all $t$
\begin{align*}
\left|\frac{P^t(u,v)}{\pi(v)} -1 \right| \leq C\frac{n}{\sqrt{t }}.
\end{align*}

\end{lemma}

\begin{proof}[\bf Proof]

For all $a>0$, if $t\geq \tunif(a)$, then it follows that 
\[
\left|\frac{P^t(u,v)}{\pi(v)} -1 \right|\leq a.
\]
Using Theorem~\ref{mixingthe} we have that in the general Eulerian case
\[
\tunif(a) \leq  \frac{c_2}{a} \left( (mn)\wedge \frac{m^2}{a} \right).
\]
Hence in order for this bound to be smaller than $t$, we need to take
\[
a = \frac{c_2(mn)}{t} \wedge \frac{\sqrt{c_2}\cdot m}{\sqrt{t}}.
\]
When $G$ is regular, then  $\tunif(a)\leq c_3 n^2/a^2$, and hence we need to take $a=\sqrt{c_3}\cdot n/\sqrt{t}$.\end{proof}

\begin{proof}[\bf Proof of Corollary~\ref{movingtarget}]

Let $X$ be a random walk on $G$ and $u=(u_s)_{s\geq 0}$ a deterministic trajectory in $G$. 
We set for all $t$
\[
Z_t(u) = \sum_{s=1}^{t} \frac{\1(X_s=u_s)}{\pi(u_s)}.
\]
From Theorem~\ref{mixingthe} we know that $\tunif(1/4)\leq C mn$ for some constant $C$ and $\tunif(1/4)\leq Cn^2$ in the regular case. 
Take $t=2Cmn$ and $t=2Cn^2$ in the regular case. Then 
\[
\E{Z_t(u)} \geq t-\tunif(1/4) \geq Cmn.
\]
For the second moment of $Z_t(u)$ we have using the Markov property
\begin{align*}
\E{Z^2_t(u)} = \sum_{s\leq r\leq t} \frac{\pr{X_s=u_s, X_r=u_r}}{\pi(u_s)\pi(u_r)} &\leq \E{Z_t(u)} + 
\sum_{s\leq t} \frac{\pr{X_s=u_s}}{\pi(u_s)} \sup_f \E{Z_t(f)} \\&\leq  \E{Z_t(u)}\left( 1+ \sup_f \E{Z_t(f)}\right).
\end{align*}
Applying Lemma~\ref{lem:shorttime} we obtain for all trajectories $u$
\begin{align*}
 \sum_{s\leq t} \frac{\pr{X_s=u_s}}{\pi(u_s)} \leq \sum_{s=1}^{n^2} \left( 1 + \frac{m}{\sqrt{s}} \right) + \sum_{s=n^2}^{Cnm} \left(1 + \frac{mn}{s} \right) \lesssim (mn)\cdot (1+ \log(m/n)).
\end{align*}
In the regular case, Lemma~\ref{lem:shorttime} gives
\begin{align*}
\sum_{s\leq t} \frac{\pr{X_s=u_s}}{\pi(u_s)} \leq \sum_{s=1}^{2Cn^2} \left( 1 + \frac{n}{\sqrt{s}}\right) \lesssim n^2.
\end{align*}
Therefore, from this and the second moment method we deduce
\begin{align*}
\pr{Z_t(u)>0} \geq \frac{(\E{Z_t(u)})^2}{\E{Z_t^2(u)}} \gtrsim \frac{1}{1+\log(m/n)}
\end{align*}
and in the regular case
\[
\pr{Z_t(u)>0} \geq c_1.
\]
This shows that the first collision time $\tau_{\rm{col}}$ is stochastically dominated by $t \times \text{Geom}(P(Z_t > 0))$, so using the above and taking expectations we obtain 
\[
\E{\tau_{\rm{col}}} \lesssim (mn) \cdot (1+\log(m/n))
\]
and $\E{\tau_{\rm{col}}}\lesssim n^2$ in the regular case and this completes the proof.
\end{proof}

\textbf{Question}\quad  Can the logarithmic term in Corollary~\ref{movingtarget} be removed, i.e.,\ is $\E{\tau_{\rm{col}}}\lesssim mn$ for all Eulerian digraphs?

\section{Sensitivity of mixing}\label{sec:sensitivity}

In this section we prove Theorem~\ref{thm:robustness}. 
We start by recalling a classical result from diophantine approximation that motivates the choice of the laziness parameter~$\alpha$.

For a finite sequence $(a_k)_{k\leq n}$ the gap is defined as
\[
\mathrm{Gap}= \sup_{k\leq n} |a'_k-a'_{k+1}|,
\]
where $(a_k')$ is the sequence $(a_k)$ sorted in increasing order.

\begin{lemma}\label{lem:numbertheory}

Let $\xi$ be an irrational with continued fraction coefficients $a_i<B$ and for all $k$ let~$x_k=(k\xi)\bmod 1$.
Then there exists a constant $c$ such that for all $n$ the gap of the finite sequence $\{x_k: \, k\leq n\}$ is at most $c/n$. Moreover, for any interval $J\subseteq [0,1]$ of length $|J|=1/n$, the number among $x_1,...,x_n$ that fall in $J$ is at most $B+2$.
\end{lemma}

The key to the proof of Theorem~\ref{thm:robustness} is the variability of the number of times that the walk goes around the left cycle versus the right cycle. This variability has little or no effect if the durations of these two types of excursions are the same or rationally related. However, if the ratio is an irrational such as the golden mean, then Lemma~\ref{lem:numbertheory} together with this variability imply that the distribution of return times to the origin is much less concentrated than in the rational case. This in turn yields the mixing bounds. 

A similar connection of mixing to diophantine approximation was found by Angel, Peres and Wilson in~\cite{AngPerWil}, but the chains considered there are not related by a change of the laziness parameter.

\begin{proof}[\bf Proof of Lemma~\ref{lem:numbertheory}]

As in \cite[page 88]{Kuipers}, let $D_n$ be the discrepancy of the sequence $x_1,\ldots, x_n$, i.e.,
\[
D_n = \sup_{0\leq \alpha \leq \beta \leq 1} \left| 
\frac{\left|k\leq n: x_k \in [\alpha,\beta) \right|}{n} - (\beta -\alpha)
\right|.
\]
In particular $D_n$ is at least the largest gap in the sequence. Let $q_i$ denote the denominators in the expansion of $\xi$, see for instance \cite[page 122, line 7]{Kuipers}. So if $a_i<B$ for all $i$, then   $q_{i+1} \leq Bq_i$ for all $i$.

Given any $n$ find the largest $i$ so that $q_i \leq  n < q_{i+1}\leq B q_i$. Now \cite[inequality~(3.17)]{Kuipers} gives $D_{q_i} <2/q_i <2B/n$, and hence the largest gap in $x_1, x_2,\ldots, x_{q_i}$ is at most $2B/n$. This then also applies to the largest gap in $x_1, \ldots, x_n$.

To prove the second assertion of the lemma, using \cite[inequality~(3.17)]{Kuipers} again we get $D_{q_{i+1}} <2/{q_{i+1}}$. This means that the number of $x_j$ for $j \leq q_{i+1}$ that fall in $J$ is at most $2+q_{i+1}|J| \leq 2+B$. This concludes the proof.
\end{proof}

In this section we consider the graph $G$  described in the Introduction and depicted in Figure~\ref{fig:twocycles} below. We call $C_1$ the cycle on the left and $C_2$ the cycle on the right.

\begin{figure}[h!]
\begin{center}
\includegraphics[scale=0.7]{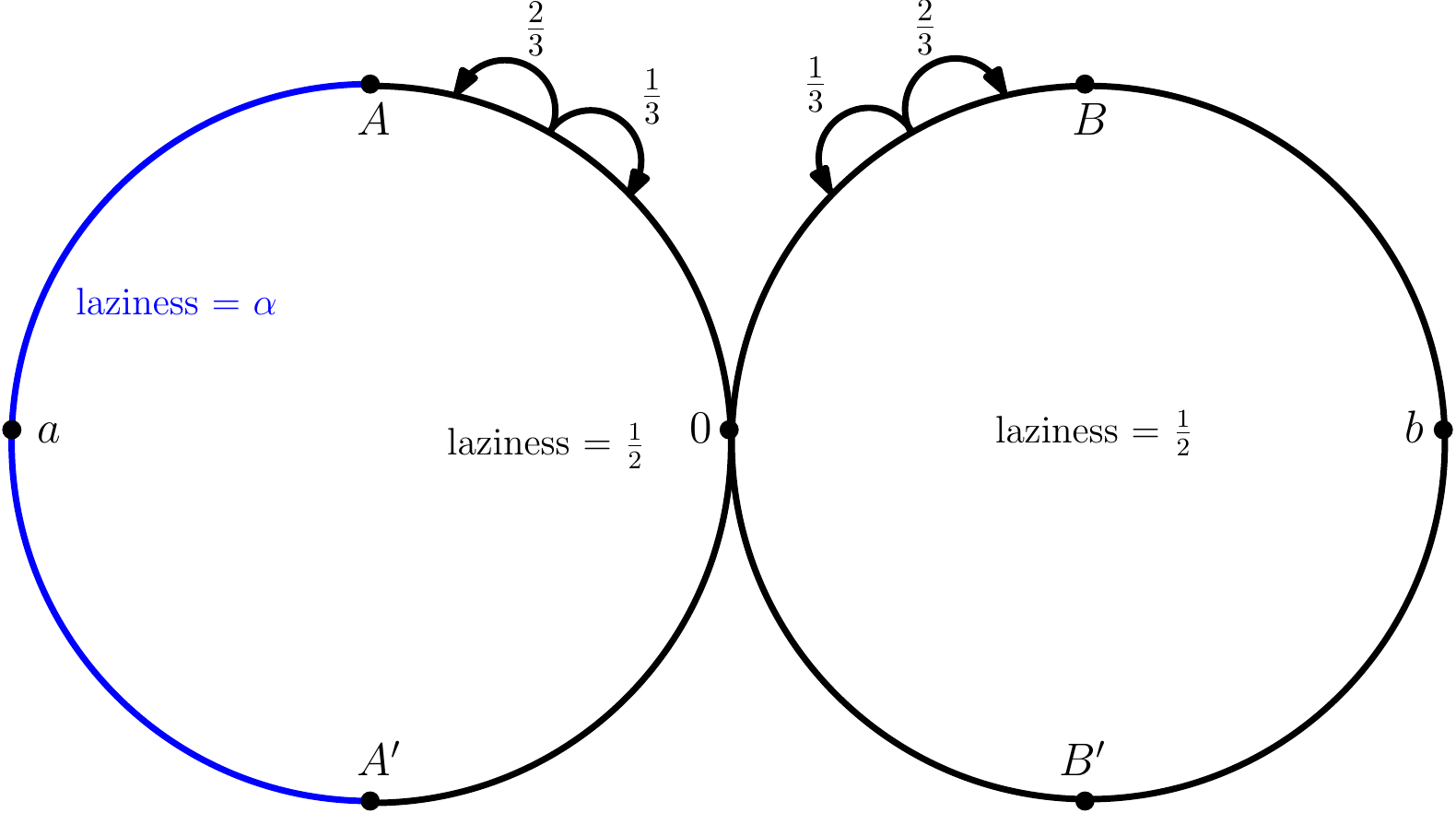}
\caption{Biased random walk on the cycles with different laziness parameters}\label{fig:twocycles}
\end{center}
\end{figure}

Let $X$ be a random walk on $G$ with transition probabilities as explained in the statement of Theorem~\ref{thm:robustness}. The key to proving Theorem~\ref{thm:robustness} is the following.

\begin{lemma}\label{lem:unifalmost}
For all $t\in \left[\frac{n^{3/2}}{50}, 10n^{3/2}\right]$ we have 
\[
\prstart{X_t = 0}{0} \asymp \frac{1}{n}.
\]
\end{lemma}

In order to prove this lemma we will construct a coupling between the walk $X$ and another walk~$Y$ which is constructed by taking independent excursions in every cycle and every time it visits $0$ it chooses one of the two cycles equally likely.  Then we will show that the coupling succeeds with high probability. 

\begin{definition}\label{def:xi}\rm{
Let  $(\xi_i)_{i\geq 1}$ be i.i.d.\ random variables taking values in $\{0,1\}$ equally likely. Let $(T_i^1)_{i\geq 1}$ be i.i.d.\ random variables distributed as the commute time between $0$ and $a$ for a random walk on the cycle $C_1$. Similarly, let $(T_i^2)_{i\geq 1}$ be i.i.d.\ random variables distributed as the commute time between $0$ and $b$ for a random walk on the cycle $C_2$.

For all $k\geq 1$ we let 
\[
S_k=\sum_{i=1}^{k} (\xi_i T_i^1+ (1-\xi_i)T_i^2).
\]
}
\end{definition}

The main ingredient in the proof of Lemma~\ref{lem:unifalmost} is the following estimate. 

\begin{lemma}\label{lem:mainingredient}
For all times $t\in \left[ \frac{n^{3/2}}{100}, 10 n^{3/2}\right]$ we have 
\[
\sum_k \pr{S_k=t} \asymp \frac{1}{n}.
\]
\end{lemma}

Before proving Lemma~\ref{lem:mainingredient} we collect some standard estimates for sums of independent random variables.

The first one is a direct application of the Local CLT for the Bernoulli random variables $(\xi_i)$.

\begin{claim}\label{cl:lcltxi}
There exists a positive constant $c$ so that for all $x\in \{-k/2,\ldots, k/2\}$ we have
\begin{align}\label{eq:clleq}
\pr{\sum_{i=1}^{k}\xi_i -\frac{k}{2}=x} \lesssim \left(\frac{1}{\sqrt{k}} \exp\left(-\frac{cx^2}{k} \right)\right)\vee \exp\left( -c\sqrt{k}\right)
\end{align}
and for all $x$ with $|x|\leq \sqrt{k}$ we have
\begin{align}\label{eq:claimasymp}
\pr{\sum_{i=1}^{k}\xi_i -\frac{k}{2}=x} \asymp \frac{1}{\sqrt{k}}.
\end{align}
\end{claim}

\begin{proof}[\bf Proof]
The claim follows using~\cite[Theorem~2.3.11]{LawlerLimic} for $|x|\leq k^{3/4}$ and the Azuma-Hoeffding inequality for $|x|>k^{3/4}$.
\end{proof}

\begin{lemma}\label{lem:lcltdi}
For all $i$ we set $D_i^1 = T_i^1 - \E{T_i^1}$ and $D_i^2 = T_i^2 - \E{T_i^2}$.
There exists a positive constant $c_1$ so that for all $k,\ell\leq n$ and all $y\in \R$ we have 
\[
\pr{\sum_{i=1}^{k}D_i^1 + \sum_{i=1}^{\ell} D_i^2 = y} \lesssim  \left(\frac{1}{\sqrt{(k+\ell)n}} \exp\left(-c_1 \frac{y^2}{(k+\ell)n } \right) \right)\vee \exp\left(-c_1 n^{1/3} \right).
\]
Moreover, for all $|y|\leq \sqrt{(k+\ell)n}$ (in the right support) we have 
\[
\pr{\sum_{i=1}^{k}D_i^1 + \sum_{i=1}^{\ell} D_i^2 = y}\asymp \frac{1}{\sqrt{(k+\ell)n}}.
\]
\end{lemma}

\begin{proof}[\bf Proof]

We first note that it suffices to prove that for all $y$ we have 
\begin{align}\label{eq:di1goal}
\begin{split}
	\pr{\sum_{i=1}^{k}D_i^1 = y} &\lesssim   \left(\frac{1}{\sqrt{kn}} \exp\left(-c_1 \frac{y^2}{kn} \right) \right)\vee 	 \exp\left(-c_1 (kn)^{1/3} \right) \quad \text{and } \\
	\pr{\sum_{i=1}^{k}D_i^1 = y} &\asymp \frac{1}{\sqrt{kn}} \quad \text{if } |y|\leq \sqrt{kn}
 	\end{split}
\end{align}
for any value of the laziness parameter $\alpha$ on the left half of the cycle. Then the same bound will also hold for $(D_i^2)$ by taking $\alpha=1/2$. From these two bounds and using that the convolution of Gaussian densities is another Gaussian proves the lemma. So we now focus on proving~\eqref{eq:di1goal}.

We split the time $T_i^1$ into the time $S_i$ to go from $0$ to $a$ and the time $S_i'$ to go from $a$ to $0$. Note that $S_i$ and $S_i'$ are independent. 

We start by analysing $S_i$. Consider a biased $(2/3, 1/3)$ random walk  $Z$ on $\Z$ with laziness equal to~$\alpha$ in $[n/4, n/2]\cup [-n/2, -n/4]$ and equal to $1/2$ everywhere else. Then the time $S_i$ has the same distribution as the first hitting time of $\{-n/2,n/2\}$ by $Z$. Let $Y$ be another biased random walk on $\Z$ with the same transitions as $Z$. We are going to couple $Y$ and $Z$ so that they are equal up to the first hitting time of $\{-n/4,n/2\}$. If they hit $-n/4$ before $n/2$, then we declare that the coupling has failed. We then get that for a positive constant $c$ we have
\[
\pr{\text{coupling has failed}}\leq e^{-cn}.
\]
On the event that the coupling has succeeded, we get that $T_{\{-n/2, n/2\}}^Z = T_{n/2}^Y$. Therefore, up to exponentially small in $n$ probability, we can replace the sum $\sum_{i=1}^{k}S_i$ by $\sum_{i=1}^{k} \tau_i$, where $(\tau_i)$ are i.i.d.\ distributed as the first hitting time $\tau$ of $n/2$ for the walk $Y$ on $\Z$, i.e.
\begin{align}\label{eq:couplingexp}
\pr{\sum_{i=1}^{k}(S_i-\E{S_i}) =y} \lesssim e^{-cn} + \pr{\sum_{i=1}^{k}(\tau_i-\E{\tau_i}) = y - O(e^{-cn})},
\end{align}
where the last step follows from the fact that $\E{T^Z_{\{-n/2,n/2\}}} = \E{T_{n/2}^Y} + O(e^{-cn})$.

Exactly in the same way as above, we get
\[
\pr{\sum_{i=1}^{k}(S_i'-\E{S_i'}) =y} \lesssim e^{-cn} + \pr{\sum_{i=1}^{k}(\tau_i'-\E{\tau_i'}) = y - O(e^{-cn})},
\]
where $\tau_i'$ is now the first hitting time of $n/2$ for a biased walk as above with the only difference being that the laziness is equal to $\alpha$ in $[-n/4, n/4]$ and equal to $1/2$ everywhere else. 

We now turn to bound the probability appearing on the right hand side of~\eqref{eq:couplingexp}. First we decompose~$\tau$ as $\tau = \sum_{j=1}^{n/2-1}T_{j-1,j}$, where $T_{j-1,j}$ is the time to hit $j$ starting from $j-1$. Note that these times are independent, but not identically distributed for all~$j$. However, for $j\leq n/8$, up to an exponentially small in $n$ probability, we can couple as before $T_{j-1,j}$ with i.i.d.\ random variables each with the distribution of the first hitting time of $1$ starting from $0$ for a biased $(2/3, 1/3)$ random walk on~$\Z$ with laziness equal to $1/2$ everywhere. Then the coupling fails only if the walk starting from $j-1$ hits $-n/4$ before $j$, which has exponentially small in $n$ probability. Similarly for $j\in [3n/8, n/2]$ we can couple them to i.i.d.\ random variables each with the same distribution as the first hitting time of $1$ starting from~$0$ by a biased random walk on $\Z$ with laziness equal to $\alpha$ everywhere. 

 Finally, for $j\in [n/8+1, 3n/8]$ the random variables $T_{j-1,j}$ are independent and all have exponential tails with the same uniform constant, i.e.\ there exist positive constants $c_1$ and $c_2$ so that for all $j$ in this range
\[
\pr{T_{j-1,j}\geq x} \leq c_1 e^{-c_2 x}.
\] 
Therefore, taking the sum over all $i\leq k$ we get by~\cite[Theorem~15]{Petrov} and Chernoff's bound for a positive constant $c_1$ and $x>0$
\begin{align}\label{eq:cltpart}
\pr{\left|\sum_{i=1}^{k}\sum_{j=n/8+1}^{3n/8} (T_{j-1,j}^{(i)} - \E{T^{(i)}_{j-1,j}})\right| \geq x} &\lesssim \exp\left(-c_1 \frac{x^2}{kn}\right) \quad \text{if } x\leq c_1 kn \\
\pr{\left|\sum_{i=1}^{k}\sum_{j=n/8+1}^{3n/8} (T_{j-1,j}^{(i)} - \E{T^{(i)}_{j-1,j}}) \right|\geq x}&\lesssim \exp\left(-c_1 x \right) \quad \text{if } x> c_1kn.
\end{align}
Using the local CLT~\cite[Theorem~2.3.11]{LawlerLimic} for $|y|\leq (kn)^{2/3}$ and~\cite[Theorem~15]{Petrov} for $|y|>(kn)^{2/3}$ we obtain for a positive constant $c$
\begin{align*}
	\pr{\sum_{i=1}^{k}\sum_{j=1}^{n/8}\left(T_{j-1,j}^{(i)} - \E{T_{j-1,j}^{(i)}}\right) =y - O(e^{-cn})} \lesssim \frac{1}{\sqrt{kn}} \exp\left( -\frac{c y^2}{kn}\right) \vee \exp\left(-c (kn)^{1/3} \right).
\end{align*}
For $|y|\leq \sqrt{kn}$ using again~\cite[Theorem~2.3.11]{LawlerLimic} yields
\[
\pr{\sum_{i=1}^{k}\sum_{j=1}^{n/8}\left(T_{j-1,j}^{(i)} - \E{T_{j-1,j}^{(i)}}\right) =y - O(e^{-cn})} \asymp \frac{1}{\sqrt{kn}}.
\]

The same bound (with different constants) holds for the sum over $j\in [3n/8+1,n/2]$. Using again that the convolution of two Gaussian densities is Gaussian we get that 
\begin{align*}
	&\pr{\sum_{i=1}^{k}\sum_{j=1}^{n/8}\left(T_{j-1,j}^{(i)} - \E{T_{j-1,j}^{(i)}}\right) + \sum_{i=1}^{k}\sum_{j=3n/8+1}^{n/2}\left(T_{j-1,j}^{(i)} - \E{T_{j-1,j}^{(i)}} \right) =y-O(e^{-cn})}\\
	 &\quad \quad\quad\quad\quad\lesssim \frac{1}{\sqrt{kn}}
	 \exp\left( -\frac{c y^2}{kn}\right) \vee \exp\left(-c (kn)^{1/3} \right)
\end{align*}
and for $|y|\leq \sqrt{kn}$
\[
\pr{\sum_{i=1}^{k}\sum_{j=1}^{n/8}\left(T_{j-1,j}^{(i)} - \E{T_{j-1,j}^{(i)}}\right) + \sum_{i=1}^{k}\sum_{j=3n/8+1}^{n/2}\left(T_{j-1,j}^{(i)} - \E{T_{j-1,j}^{(i)}} \right) =y-O(e^{-cn})}\asymp \frac{1}{\sqrt{kn}}.
\]
We now set 
\begin{align*}
	A_k &= \sum_{i=1}^{k}\sum_{j=1}^{n/8}\left(T_{j-1,j}^{(i)} - \E{T_{j-1,j}^{(i)}}\right) + \sum_{i=1}^{k}\sum_{j=3n/8+1}^{n/2}\left(T_{j-1,j}^{(i)} - \E{T_{j-1,j}^{(i)}} \right) \quad \text{and}\\
B_k &= \sum_{i=1}^{k}\sum_{j=n/8+1}^{3n/8} \left(T_{j-1,j}^{(i)} - \E{T^{(i)}_{j-1,j}}\right). 
\end{align*}
Putting all these estimates together we get 
\begin{align*}
	\pr{\sum_{i=1}^{k}D_i^1 = y} &= \sum_{x:|x|\geq |y|/2}\pr{B_k =x}\pr{A_k=y'-x}  + \sum_{x: |x|\leq \frac{|y|}{2}} \pr{A_k = y'-x}\pr{B_k=x} \\
	&\lesssim \frac{1}{\sqrt{kn}} \exp\left(-c_1\frac{y^2}{kn} \right)\vee \exp\left( -c_1(kn)^{1/3} \right),
\end{align*}
where we set $y'=y-O(e^{-cn})$. For $|y|\leq \sqrt{kn}$ we get
\[
\pr{\sum_{i=1}^{k}D_i^1 = y}\asymp \frac{1}{\sqrt{kn}}
\]
and this completes the proof of~\eqref{eq:di1goal}.
\end{proof}

%
%
%
%
%

%
%
%
%

To prove Lemma~\ref{lem:mainingredient} we also need to know the ratio of $\estart{T_1^1}{0}$ to $\estart{T_1^2}{0}$. This is the content of the following lemma which is a standard result, but we include the proof for the reader's convenience.

\begin{lemma}\label{lem:lazyalpha}
For all $n$ we have 
\begin{align*}
&\estart{T_1^1}{0} = \left(2+\frac{1}{1-\alpha} \right)\cdot f(n) \quad \text{and}\quad \vr{T_1^1}\asymp n, \\
&\estart{T_1^2}{0} = 4 f(n) \quad \text{and}\quad \vr{T_1^2}\asymp n,
\end{align*}
where 
\[
f(n)=\frac{3n}{2} -3n \cdot 2^{-n/2} \cdot \frac{2^{n/2} -1}{2^{n/2}-2^{-n/2}}. 
\]
\end{lemma}

\begin{proof}[\bf Proof]
It suffices to prove the statement for $T_1^1$, since taking $\alpha=1/2$ proves the second assertion.
Recall that $T_1^1$ is the time it takes starting from $0$ to hit $a$ and come back to $0$ afterwards. We split the time $T_1^1$ into the time $S_1$ to go from $0$ to $a$ and the time $S_2$ to go from $a$ to $0$. Since we are only interested in the expectation of $T_1^1$ we are going to couple $S_1$ and $S_2$ as follows: we identify the upper part of the cycle $[0,n/2]$ with the lower part $[n/2,n]$, i.e.\ we identify $i\in [0,n/2]$ with~$n/2+i \in [n/2,n]$. Let $S$ be the time it takes for a non-lazy walk to go from~$0$ to $a$. Then $S$ has the same law as the time to go from $a$ to $0$. Let $L_i$ be the local time at vertex $i$ up to time~$S$, i.e.\ the number of visits of the non-lazy walk to $i$ before hitting~$a$ starting from~$0$. Then this has the same law as $L_{n/2+i}$ for a random walk starting from $a$ before it hits $0$. Let~$(\xi_{i,j})$ and $(\til{\xi}_{i,j})$ be two independent families of~i.i.d.\ geometric random variables with means $1$ and $1/(1-\alpha) - 1$ respectively. Using the identification of the two parts of the cycle explained above, we can now write 
\begin{align*}
S_1&=S+ \sum_{i=0}^{n/4}\sum_{j=1}^{L_i} \xi_{i,j} + \sum_{i=n/4+1}^{n/2} \sum_{j=1}^{L_i}\til{\xi}_{i,j} + \sum_{i=n/2+1}^{3n/4} \sum_{j=1}^{L_i} \til{\xi}_{i,j} + \sum_{i=3n/4+1}^{n}\sum_{j=1}^{L_i}\xi_{i,j}
\\
S_2&= S + \sum_{i=0}^{n/4}\sum_{j=1}^{L_i} \til{\xi}_{i,j} + \sum_{i=n/4+1}^{n/2} \sum_{j=1}^{L_i}{\xi}_{i,j} + \sum_{i=n/2+1}^{3n/4} \sum_{j=1}^{L_i} \xi_{i,j} + \sum_{i=3n/4+1}^{n}\sum_{j=1}^{L_i} \til{\xi}_{i,j},
\end{align*}
where $n$ is identified with $0$ in $\Z_n$.
Taking expectation and adding these two equalities yields
\begin{align*}
\estart{T_1^1}{0} =2\E{S} + \sum_{i} \E{L_i} \E{\xi_{1,1}+\til{\xi}_{1,1}}=2\E{S}+\left(1+\frac{1}{1-\alpha} -1\right) \E{S} = \left( 2+\frac{1}{1-\alpha}\right)\E{S}.
\end{align*}
An elementary calculation shows that $\E{S}=f(n)$.
We now turn to prove the asymptotic for the variance. It suffices to show that $\vr{S_1}\asymp n$, since $\vr{T_1^1} = 2\vr{S_1}$. Let $Y$ be a biased $(2/3,1/3)$ random walk on $\Z$ with laziness equal to $\alpha$ in $[a/2, a]$ and $[-a,-a/2]$ and equal to $1/2$ everywhere else. Then $S_1$ has the same law as~$\tau_a\wedge \tau_{-a}$, where $\tau_x$ stands for the first hitting time of $x$ by $Y$. We can then write 
\[
\tau_a = \tau_a\wedge \tau_{-a} + \1(\tau_{-a}<\tau_a) \cdot \til{\tau},
\]
where $\til{\tau}$ is the time it takes to hit $a$ starting from $-a$, and is independent of $\1(\tau_{-a}<\tau_a)$. Note that we can write $\tau_a = \sum_{i=1}^{a}T_{i-1,i}$, where $T_{i-1,i}$ is the time it takes $Y$ to hit $i$ starting from $i-1$. Using that $(T_{i-1,i})_i$ have exponential tails and Azuma-Hoeffding we obtain that $\vr{\tau_a}\asymp a$ and $\vr{\til{\tau}} \asymp a$. This together with the above decomposition of $\tau_a$ finishes the proof.
\end{proof}

\begin{proof}[\bf Proof of Lemma~\ref{lem:mainingredient}]

For all $k$ we have 
\begin{align*}
\pr{S_k=t} = \pr{\sum_{i=1}^{k}(\xi_i T_i^1 + (1-\xi_i)T_i^2)=t} = \sum_{x=-k/2}^{k/2}\pr{\sum_{i=1}^{k}\xi_i =\frac{k}{2}+x,\sum_{i=1}^{k}(\xi_i T_i^1 + (1-\xi_i)T_i^2)=t}. 
\end{align*}
Since the i.i.d.\ families $(T_i^1)_i, (T_i^2)_i$ and $(\xi_i)_i$ are independent, by renumbering it follows that
\begin{align*}
\pr{\sum_{i=1}^{k}\xi_i =\frac{k}{2}+x,\sum_{i=1}^{k}(\xi_i T_i^1 + (1-\xi_i)T_i^2)=t} =\pr{\sum_{i=1}^{k}\xi_i=\frac{k}{2}+x} \pr{\sum_{i=1}^{k/2+x}T_i^1 + \sum_{i=1}^{k/2-x} T_i^2=t}.
\end{align*}
We set for all $i$
\begin{align*}
T_i^1 = \beta f(n) + D_i^1 \quad \text{and} \quad T_i^2 = 4 f(n) + D_i^2,
\end{align*}
where $\beta=(3-2\alpha)/(1-a)$.
We now have
\begin{align*}
\pr{\sum_{i=1}^{k/2+x}T_i^1 + \sum_{i=1}^{k/2-x} T_i^2=t} =\sum_y \1\left(\frac{k}{2} f(n) (\beta+4) + f(n)(\beta-4)x +  y=t\right)\times \\ 
 \pr{\sum_{i=1}^{k/2+x} D_i^1 + \sum_{i=1}^{k/2-x} D_i^2 =y}.
\end{align*}
Since $\E{T_i^1} = \beta f(n)$ and $\E{T_i^2}=4f(n)$, it follows that $\E{D_i^1} = \E{D_i^2} = 0$ and by the independence of $(T_i^1)$ and $(T_i^2)$, it follows that the same holds for $(D_i^1)$ and~$(D_i^2)$. Applying Lemma~\ref{lem:lcltdi} now yields for all $y$
\begin{align}\label{eq:upperclt}
\pr{\sum_{i=1}^{k/2+x} D_i^1 + \sum_{i=1}^{k/2-x} D_i^2 =y}
\lesssim \left(\frac{1}{\sqrt{kn}} \exp\left(-c_1\frac{y^2}{kn} \right)\right) \vee \exp\left(-c_1 n^{1/3}\right)
\end{align}
and for $|y|\leq \sqrt{kn}$ we have
\begin{align}\label{eq:equall}
\pr{\sum_{i=1}^{k/2+x} D_i^1 + \sum_{i=1}^{k/2-x} D_i^2 =y} \asymp \frac{1}{\sqrt{kn}}.
\end{align}
Putting everything together we obtain
\begin{align}\label{eq:sumskt}
 \nonumber\pr{S_k=t} =\sum_{x,y} \1\left(\frac{k}{2} f(n) (\beta+4) + f(n)(\beta-4)x +  y=t\right)\pr{\sum_{i=1}^{k}\xi_i = \frac{k}{2}+x} \times\\\pr{\sum_{i=1}^{k/2+x} D_i^1 + \sum_{i=1}^{k/2-x} D_i^2 =y}.
\end{align}
To simplify notation we set $g(k,x,y)$ for the expression appearing in the sum above. 
We start by proving the lower bound, i.e.
\begin{align}\label{eq:goallower}
\sum_k \pr{S_k=t} \gtrsim \frac{1}{n}.
\end{align}
To do this, we will choose specific values for $k, x$ and $y$ to satisfy the indicator. Dividing through by $f(n)(\beta+4)/2$ we want to find $k, x$ and $y$ to satisfy the equation
\begin{align*}
k + \frac{2(\beta-4)}{\beta+4}x + \frac{2y}{f(n) (\beta+4)}= \frac{2t}{f(n)(\beta+4)}.
\end{align*}
Consider the set
\[
\left\{\ell \cdot \frac{2(\beta-4)}{\beta+4} : \,\,|\ell|\leq n^{1/4}\right\}.
\]
Then by Lemma~\ref{lem:numbertheory} there exists $x$ with $|x|\leq n^{1/4}$ such that 
\[
\left(\frac{2(\beta-4)}{\beta+4} \cdot x - \frac{2t}{f(n)(\beta+4)}\right)\bmod 1 \leq c\cdot n^{-1/4},
\]
where $c$ is the constant from Lemma~\ref{lem:numbertheory} for the irrational number~$2(\beta - 4)/(\beta +4)$, which has bounded continued fraction coefficients.
Now take $k$ to be the closest integer to $\frac{2(\beta-4)}{\beta+4} \cdot x - \frac{2t}{f(n)(\beta+4)}$. Then the above inequality gives that 
\[
\left|\frac{2(\beta-4)}{\beta+4} \cdot x - \frac{2t}{f(n)(\beta+4)} - k\right|\leq  c\cdot n^{-1/4}.
\]
From this, if we take $y=t - k f(n) (\beta+4) /2 - f(n) (\beta -4)x$ we see that $|y|\leq c' n^{3/4}$ for some positive constant~$c'$. Also we get that $k\asymp \sqrt{n}$. Plugging these values for $k, x$ and $y$ in~\eqref{eq:sumskt} and using~\eqref{eq:equall} immediately proves~\eqref{eq:goallower}.

It remains to prove the upper bound, i.e.\ using the notation introduced above
\begin{align}\label{eq:goalupper}
\sum_k\pr{S_k=t} = \sum_{k,x,y} g(k,x,y) \lesssim \frac{1}{n}.
\end{align}
We first notice that since $t\asymp n^{3/2}$ in the above sum we only take $k\leq c\sqrt{n}$ for a positive constant~$c$. We now split the sum into the following sets defined for $
\ell, m \in \Z$
\[
A_{\ell, m} = \{ k\leq c\sqrt{n}, \ell n^{1/4}\leq x \leq (\ell+1)n^{1/4}, m n^{3/4} \leq y \leq (m+1)n^{3/4}\}.
\] 
We are going to use that if $c_2$ is a positive constant, then for all $n$ large enough we have
\begin{align}\label{eq:allcases}
\forall x:\, |x|\leq n^{1/4}, \,\, \exists\,\, \text{at most one}  \,\,k: \,\left|k+ \frac{2(\beta-4)}{(\beta+4)} x -\frac{2t}{f(n)(\beta+4)}\right| \leq c_2 n^{-1/4}.
\end{align}
For all $x$ with $|x|\leq n^{1/4}$ from~\eqref{eq:allcases} there exists at most one~$k$, and hence at most one~$y$ such that 
\begin{align}\label{eq:indkxy}
\frac{k}{2}f(n) (\beta+4) + f(n) (\beta -4)x + y=t.
\end{align}
Since $|x|\leq k/2$, we get that $|\ell|\leq c_3 n^{1/4}$ and the above equality  gives that $|m|\leq c_4 n^{3/4}$, where $c_3$ and $c_4$ are two positive constants.

Fix now $\ell, m$ as above and let $(k,x,y), (\til{k}, \til{x}, \til{y}) \in A_{\ell,m}$ both satisfy~\eqref{eq:indkxy}. Then 
\begin{align*}
\left|(k-\til{k}) + \frac{2(\beta-4)}{\beta+4}(x-\til{x})\right| \leq c_5 n^{-1/4},
\end{align*}
where $c_3$ is a positive constant.
Since $|x-\til{x}|\leq n^{1/4}$, Lemma~\ref{lem:numbertheory} gives that the number of such triples is of order $1$. Note that for any $a>0$ the function $\frac{e^{-a/x}}{x}$ is decreasing for $x<a$ and increasing for $x>a$. Therefore for $k\leq  c\sqrt{n}$ we get that for all $r\in \Z$
\begin{align}\label{eq:expbound}
	\frac{1}{k} \exp\left(-\frac{cr^2\sqrt{n}}{k} \right) \leq \frac{1}{\sqrt{n}} e^{-c r^2}.
\end{align}
Using this, Claim~\ref{cl:lcltxi} and~\eqref{eq:upperclt} we get 
\begin{align*}
	\sum_{|\ell|\leq c_3 n^{1/4}}\sum_{|m|\leq c_4 n^{3/4}} \sum_{(k,x,y)\in A_{\ell,m}} g(k, x, y) \lesssim \sum_{\ell \in \Z} \sum_{m\in \Z} \frac{1}{n} \exp\left(-c(\ell^2+m^2) \right) + \exp\left(-c_6 n^{1/3} \right)\asymp \frac{1}{n}
\end{align*}
and this concludes the proof of the lemma.
\end{proof}

The next result provides a coupling between the walk $X$ and another walk $Y$ on the graph $G$ that evolves as follows. 

\begin{definition}\label{def:walkY}\rm{
We define a round for $X$ as the time elapsed between a visit to $0$ and a return to $0$ after having first hit either $a$ or $b$. Recall that $(\xi_i)_i$ are i.i.d.\ Bernoulli random variables with mean $1/2$. 
At the beginning of the $i$-th round, if $\xi_i=0$, then $Y$ walks on $C_1$ up until the time it hits $0$ after having first visited $a$, i.e.\ $Y$ spends time~$T_i^1$ on $C_1$. If $\xi_i=1$, then $Y$ walks on $C_2$ up until hitting $0$ having first visited $b$, i.e.\ it spends time $T_i^2$ on $C_2$. Let $\sigma:G\to G$ be a mapping such that every $x\in C_1$ is mapped to the corresponding $x\in C_2$ and vice versa.}
\end{definition}

For a Markov chain $W$ we write $\tau_S(W)$ for the first hitting time of the set $S$, i.e.\
\[
\tau_S(W)=\inf\{t\geq 0: \, W_t\in S\}.
\]
If $S=\{x\}$, then we simply write $\tau_x(W)$. If the chain under consideration is clear from the context, then we simply write $\tau_S$.

\begin{lemma}\label{lem:coupling}
There exists a coupling of $X$ and $Y$ so that 
\[
\pr{\forall \, t\leq n^2:\, X_t = Y_t \,\text{ or }\, X_t=\sigma(Y_t)} \geq 1-ce^{-cn},
\]
where $c$ is a positive constant. 
\end{lemma}

\begin{proof}[\bf Proof]
 
Let $T_0$ be the first time that $X$ hits $0$ after having first hit either $a$ or $b$. We will construct a coupling of $X$ and $Y$ so that 
\begin{align}\label{eq:coupling}
\pr{\forall \, t\leq T_0:\, X_t=Y_t \,\text{ or } \, X_t =\sigma(Y_t)} \geq 1-ce^{-cn}
\end{align}
for a positive constant $c$. If the coupling succeeds in the first round, then we take an independent copy of the same coupling in the second round and so on. Since by time $n^2$, there are at most $n$ rounds, the union bound and~\eqref{eq:coupling} prove the lemma. 

To lighten the notation we write $\tau=\tau_{\{A,A',B,B'\}}$.
 Let $\zeta\in \{0,1\}$ be a fair coin independent of~$X$. If $\zeta=0$, then for $t\leq T_0$ we set 
\begin{align*}
Z_t =\begin{cases}
 X_t \quad &\mbox{if} \,\,\,X_t \in C_1 \\
\sigma(X_t)  \quad &\mbox{if} \,\,\, X_t \in C_2.
\end{cases}
\end{align*}
If $\zeta=1$, then for $t\leq T_0$ we set
\begin{align*}
Z_t =\begin{cases}
 X_t \quad &\mbox{if} \,\,\,X_t \in C_2 \\
\sigma(X_t)  \quad &\mbox{if} \,\,\, X_t \in C_1.
\end{cases}
\end{align*}
From this construction it follows that $Z_{\tau}\in \{X_\tau, \sigma(X_\tau)\}$. So now we define $Y$ by setting for $t\leq \tau$
\begin{align*}
Y_t = \begin{cases}
Z_t \quad &\mbox{if} \,\,\, X_\tau=Z_\tau \\
 \sigma(Z_t)  \quad &\mbox{otherwise}.
 \end{cases}
\end{align*}
It is easy to see that the event $\{X_\tau = Z_\tau\}$ and the random walk path $X$ are independent. Indeed for a set of paths $F$ we have
\begin{align*}
\pr{X_\tau= Z_\tau, (X_t) \in F} &=\pr{X_\tau, Z_\tau\in \{A,A'\}, (X_t) \in F} +\pr{X_\tau, Z_\tau\in \{B,B'\}, (X_t) \in F} \\
&= \pr{X_\tau\in \{A,A'\},\zeta=0, (X_t) \in F} + \pr{X_\tau\in \{B,B'\}, \zeta=1, (X_t) \in F} \\
&= \frac{1}{2} \pr{(X_t) \in F} = \pr{X_\tau=Z_\tau} \cdot \pr{(X_t) \in F}.
\end{align*}
Given the independence, it is straightforward to check that $Y$ has the same distribution as the random walk path~$X$ on the two cycles up to time~$\tau$. 

It also follows immediately from the construction that $Y_\tau=X_{\tau}$ and at each time $t\leq \tau$ we have either $Y_t=X_t$ or $Y_t=\sigma(X_t)$. 

After time $\tau$ we continue $X$ and $Y$ together until they hit $a$ if $Y_\tau\in \{A,A'\}$ or until they hit $b$ if $Z_\tau\in \{B,B'\}$. If they hit $0$ before hitting $a$ or $b$, then we continue them independently afterwards and we declare that the coupling has failed. After hitting either $a$ or $b$ we continue them together until they hit $0$. This completes a round. Note that the probability the coupling fails in one round is by symmetry at most
\begin{align*}
\prstart{\tau_{A'}<\tau_A}{0}+
\prstart{\tau_0<\tau_a}{A} \leq 2^{-n/2+1}
\end{align*}
and this proves~\eqref{eq:coupling} and finishes the proof of the lemma.
\end{proof}
%

\begin{claim}\label{cl:exp}
There exists a positive constant $c$ such that for all $s\geq 0$ we have
\[
 \prstart{X_{s}=0, \tau_a\wedge \tau_b>s}{0} \lesssim e^{-cs} + e^{-cn}.
\]
\end{claim}

\begin{proof}[\bf Proof]
Writing again $\tau =\tau_{\{A,A',B,B'\}}$, we have
\begin{align*}
\prstart{X_s=0, \tau_a\wedge \tau_b>s}{0} = &\prstart{X_s=0, \tau_a\wedge \tau_b>s, \tau >s}{0} \\
&+ \prstart{X_s=0, \tau_a\wedge \tau_b>s,\tau \leq s}{0}.
\end{align*}
For the second probability appearing above using symmetry we immediately get
\begin{align*}
	\prstart{X_s=0, \tau_a\wedge \tau_b>s, \tau\leq s}{0} \leq 2(\prstart{\tau_{A'}<\tau_A}{0}+\prstart{\tau_0<\tau_a}{A})\leq 2^{-n/2+2}.
\end{align*}
 Let $\pi$ be the projection mapping, i.e.\ $\pi(x) = x$ if $x\in C_1$ and if $y\in C_2$, then it maps it to the corresponding point on $C_1$. Then clearly $Y=(\pi(X_t))_{t\leq \tau}$ is a lazy $(1/2)$ biased random walk on~$C_1$. We now get 
\begin{align}\label{eq:2n2}
\prstart{X_s=0, \tau_a\wedge \tau_b>s, \tau >s}{0} \leq \prstart{Y_s=0, \tau>s}{0}.
\end{align}
Let $\til{Y}$ be a lazy biased random walk on $\Z$ with $p_{i,i+1}=1-p_{i,i-1}=2/3$. Then 
\begin{align}\label{eq:couple}
\prstart{Y_s=0,\tau >s}{0} \leq \prstart{\til{Y}_s=0}{0} \leq \prstart{\left|\til{Y}_s -s/6 \right|\geq s/6}{0}.
\end{align}
Since $\til{Y}_s-s/6$ is a martingale with increments bounded by $7/6$, applying Azuma-Hoeffding's inequality gives
\[
\prstart{\left|\til{Y}_s -s/6 \right|\geq s/6}{0} \leq 2e^{-cs},
\]
where $c$ is a positive constant. This together with~\eqref{eq:2n2} finishes the proof.
\end{proof}

\begin{proof}[\bf Proof of Lemma~\ref{lem:unifalmost}]

Recall starting from $0$ a round for $X$ was defined to be the time passed until a return to $0$ after having first visited either $a$ or $b$.

Let $L$ be the last time before time $t$ the walk completed a round. If no round has been completed before time $t$, then we take $L=0$. Hence in all cases $X_L=0$ and we obtain
\begin{align*}
\prstart{X_t=0}{0} = \sum_{t'<t} \prstart{L=t', X_t=0}{0} &= \sum_{t'<t} \prstart{L=t'}{0} \prstart{X_{t-t'}=0, \tau_a\wedge \tau_b>t-t'}{0}\\
&\lesssim e^{-cn} + \sum_{t'<t} \prstart{L=t'}{0} e^{-c(t-t')},
\end{align*}
where for the last inequality we used Claim~\ref{cl:exp}. Therefore we get
\begin{align}\label{eq:upperbound}
\nonumber & \sum_{t'<t} \prstart{L=t'}{0} e^{-c(t-t')} \leq  \sum_{\substack{t'<t \\ t-t'>2\log n/c}}\prstart{L=t'}{0}  e^{-c(t-t')} + \sum_{\substack{t'<t \\ t-t'\leq 2\log n/c}} \prstart{L=t'}{0} e^{-c(t-t')} \\
&\leq \frac{1}{n^2} + \sum_{\substack{t'<t \\ t-t'\leq 2\log n/c}} \prstart{L=t'}{0} e^{-c(t-t')}. 
\end{align}
For the independent random variables $(T_i^1)$, $(T_i^2)$ and $(\xi_i)$ from Definition~\ref{def:xi}, we now define
\begin{align*}
N = \max\left\{ k\geq 0: \sum_{i=1}^{k} \left(\xi_i T_i^1 + (1-\xi_i) T_i^2 \right) \leq  t\right\}
\end{align*}
and as in Definition~\ref{def:xi} we write $S_k$ for the sum appearing in the maximum above. 

For the term $\prstart{L=t'}{0}$ we have
\begin{align*}
\prstart{L=t'}{0}= \prstart{L=t', \text{coupling succeeded}}{0} + \prstart{L=t', \text{coupling failed}}{0},
\end{align*}
where by success we mean that $\{X_t=Y_t \text{ or } X_t = \pi(Y_t),\, \forall t\leq n^2\}$.
Therefore from Lemma~\ref{lem:coupling} we get 
\begin{align*}
\pr{S_N=t'} \lesssim \prstart{L=t'}{0} \leq \pr{S_N=t'}+ c e^{-cn}.
\end{align*}
By the independence of the random variables~$N$ and~$S_k$ we obtain
\begin{align*}
\pr{S_N=t'} &= \sum_{k} \pr{N=k, S_k=t'} = \sum_k \pr{S_k=t', \xi_{k+1} T_{k+1}^1 + (1-\xi_{k+1}) T_{k+1}^2 >t-t'} \\
&= \sum_k \pr{S_k=t'} \pr{\xi_{k+1} T_{k+1}^1 + (1-\xi_{k+1}) T_{k+1}^2 >t-t'}.
\end{align*}
By Chebyshev's inequality and the concentration of $T_i^1$ and $T_i^2$ around their means by~Lemma~\ref{lem:lazyalpha} we get for $t-t'\leq 2\log n/c$ and $n$ large enough
\begin{align*}
\pr{\xi_{k+1} T_{k+1}^1 + (1-\xi_{k+1}) T_{k+1}^2 >t-t'} \geq \frac{1}{2}.
\end{align*}
Therefore, this implies  that for $t'$ with $t-t'\leq 2\log n/c$ and all $n$ large enough
\[
\pr{S_N=t'}\asymp \sum_k \pr{S_k=t'} \asymp \frac{1}{n},
\]
where the last equivalence follows from Lemma~\ref{lem:mainingredient}. 
Putting all estimates together in~\eqref{eq:upperbound} yields
\[
\prstart{X_t=0}{0}\lesssim \frac{1}{n}.
\]
Clearly we also have 
\begin{align*}
\prstart{X_t=0}{0} \geq \prstart{L=t}{0} \gtrsim \pr{S_N=t} \asymp\sum_k \pr{S_k=t} \asymp \frac{1}{n}
\end{align*}
and this concludes the proof of the lemma.
\end{proof}

\begin{lemma}\label{lem:0x}
For all times $t\in \left[\frac{n^{3/2}}{25}, 10n^{3/2}\right]$ and all $x\in C_1\cup C_2$ we have 
\[
\prstart{X_t=x}{0} \asymp \frac{1}{n}.
\]
\end{lemma}

\begin{proof}[\bf Proof]

Let $L$ be as in the proof of Lemma~\ref{lem:unifalmost} and let $T_0$ be the first time that $X$ completes a round as in the proof of Lemma~\ref{lem:coupling}. Then we have 
\begin{align*}
\prstart{X_t=x}{0} &= \sum_{t'<t} \prstart{L=t', X_t=x}{0} \leq 
 \sum_{\substack{t'<t\\ t-t'<2\estart{T_0}{0}}} \prstart{X_{t'}=0}{0} \prstart{T_0>t-t', X_{t-t'}=x}{0} \\&+ \sum_{\substack{t'<t\\ t-t'\geq 2\estart{T_0}{0}}}\prstart{L=t'}{0}  \prstart{T_0>t-t', X_{t-t'}=x}{0} = A_1+A_2.
\end{align*}
When $t'$ satisfies $t-t'<2\estart{T_0}{0}$, then using the range of values that $t$ can take, we get that for~$n$ large enough, $t'\in \left[\frac{n^{3/2}}{50}, 10n^{3/2}\right]$, and hence Lemma~\ref{lem:unifalmost} applies and yields $\prstart{X_{t'}=0}{0}\asymp 1/n$. Therefore, 
\begin{align*}
A_1\asymp  \frac{1}{n} \sum_{s<2\estart{T_0}{0}}\prstart{T_0>s, X_s=x}{0} \leq \frac{1}{n} \estart{\sum_{s=1}^{T_0} \1(X_s=x)}{0}.
\end{align*}
Since $T_0$ is a stopping time satisfying $\prstart{X_{T_0}=0}{0}=1$, using~\cite[Lemma~10.5]{LevPerWil} we get 
\[
\estart{\sum_{s=1}^{T_0} \1(X_s=x)}{0} = \pi(x) \estart{T_0}{0} =c_1,
\]
where $c_1$ is a positive constant. We now claim that $\E{T_0}\asymp n$ and $\vr{T_0}\asymp n$. Indeed, using the coupling of Lemma~\ref{lem:coupling} with probability $1-ce^{-cn}$ the time $T_0$ is equal to $T_1^1$ or~$T_1^2$ equally likely. This together with Lemma~\ref{lem:lazyalpha} justifies the claim.

Therefore, for the quantity $A_2$ we have
\begin{align*}
A_2 \leq  \sum_{\substack{t'<t\\ t-t'\geq 2\estart{T_0}{0}}} \prstart{L=t'}{0} \prstart{T_0\geq 2\estart{T_0}{0}}{0} \leq \prstart{T_0\geq 2\estart{T_0}{0}}{0} \leq 
 \frac{\vr{T_0}}{(\estart{T_0}{0})^2} \asymp \frac{1}{n},
\end{align*}
where the last inequality follows from Chebyshev's inequality. 

We now turn to prove a lower bound. Using Lemma~\ref{lem:unifalmost} again we have for $n$ large enough
\begin{align*}
\prstart{X_t=x}{0} &\geq \sum_{\substack{t'<t\\ t-t'<3\estart{T_0}{0}}} \prstart{X_{t'}=0}{0} \prstart{T_0>t-t', X_{t-t'}=x}{0} \\&\asymp \frac{1}{n} \estart{\sum_{s=1}^{T_0\wedge 3\estart{T_0}{0}} \1(X_s=x)}{0} \geq \frac{1}{n}\estart{\sum_{s=1}^{T_0}\1(X_s=x) \1(T_0<3\estart{T_0}{0}, \tau_x<T_0)}{0} \\
&\geq \frac{1}{n} \prstart{T_0<3\estart{T_0}{0}, \tau_x<T_0}{0} \asymp \frac{1}{n}.
\end{align*}
For the last equivalence we used that as $n\to \infty$
\[
\prstart{T_0<3\estart{T_0}{0}}{0} \geq \frac{
2}{3}
\quad \text{and} \quad \prstart{\tau_x<T_0}{0} \geq \frac{1}{2}(1-o(1)),
\]
where the first inequality follows from Markov's inequality and for the second one we used that with probability $1/2$ the walk goes around the cycle where $x$ belongs and the probability that $x$ is not hit is at most $2^{-n/2}$. This completes the proof.
\end{proof}

\begin{lemma}\label{lem:allxy}
For all $x,y$ and all times $t \in \left[\frac{n^{3/2}}{10}, 10n^{3/2}\right]$ we have 
\[
\prstart{X_t=y}{x} \asymp \frac{1}{n}.
\]
\end{lemma}

\begin{proof}[\bf Proof]

We have
\begin{align*}
\prstart{X_t=y}{x} = \sum_{s<t} \prstart{\tau_0=s, X_t=y}{x} + \prstart{\tau_0>t, X_t=y}{x}.
\end{align*}
Note that $\prstart{\tau_0>t}{x} \leq 1/n^2$ by Chebyshev's inequality. By the Markov property we now get
\begin{align*}
&\sum_{s<t} \prstart{\tau_0=s, X_t=y}{x} = \sum_{s<t} \prstart{\tau_0=s}{x}\prstart{ X_{t-s}=y}{0} \\
&= \sum_{s<2\estart{\tau_0}{x}}\prstart{\tau_0=s}{x}\prstart{ X_{t-s}=y}{0} +  \sum_{2\estart{\tau_0}{x}<s<t}\prstart{\tau_0=s}{x}\prstart{ X_{t-s}=y}{0}.
\end{align*}
By Chebyshev's inequality again we get
\[
\prstart{\tau_0\geq 2\estart{\tau_0}{x}}{x} \lesssim\frac{1}{n}.
\]
When $t\in \left[\frac{n^{3/2}}{10}, 10n^{3/2}\right]$ and $s<2\estart{\tau_0}{x} \lesssim n$, then $t-s \in \left[\frac{n^{3/2}}{25}, 10n^{3/2} \right]$ for $n$ large enough. Hence Lemma~\ref{lem:0x} gives for $n$ large enough
\[
\prstart{X_t=y}{x} \asymp \frac{1}{n}
\]
and this completes the proof.
\end{proof}

\begin{lemma}\label{lem:linfdist}
Let $P$ be a transition matrix on a finite state space with stationary distribution $\pi$. Let $d^\infty$ denote the $\LL^\infty$ distance $d^1$ the total variation distance, i.e.,
	\[
	d^\infty(r) =  \max_{x,y} \left| \frac{P^r(x,y)}{\pi(y)} -1\right| \quad \text{and} \quad d^1(r) = \max_x \|P^r(x,\cdot) - \pi\|_{\rm{TV}}.
	\]
Then for all $s,t\in \N$ we have
	\[
	d^{\infty}(s + t) \leq d^{\infty}(s) d^1(t).
	\]
\end{lemma}

\begin{proof}[\bf Proof]
	A standard duality argument (see for instance Chapter~1 of Saloff-Coste~\cite{SaloffCoste}) implies that for all $p\in [1,\infty]$, the $\LL^p$ distance to stationarity at time $t$ equals the operator norm $\|P^t\|_{q \to \infty}$. Here~$P^t$ acts on the functions that have mean zero w.r.t.\ $\pi$ and $1/p+1/q=1$. Clearly  
	\[
	\|P^{t+s}\|_{q \to \infty} \le \|P^t\|_{q \to \infty}  \|P^s\|_{\infty \to \infty}.
	\]
	Applying this with $p=\infty$ and $q=1$ gives what we want.

\end{proof}

\begin{proof}[\bf Proof of Theorem~\ref{thm:robustness}](upper bound)

Let $t=n^{3/2}$. By Lemma~\ref{lem:allxy} there exist two positive constants $c_1<1$ and $c_2>1$ so that for all $x$ and $y$ we have 
\begin{align}\label{eq:lem12}
c_1\pi(y) \leq P^t(x,y)\leq c_2\pi(y).
\end{align}
Therefore, for any two vertices $x,x'$ on the graph $G$ we get
\begin{align*}
\|P^{t}(x,\cdot) - P^t(x',\cdot)\|_{\rm{TV}} = 1 - \sum_{y} P^t(x,y)\wedge P^{t}(x',y) \leq 1-c_1.
\end{align*}
Since $\bar{d}(s) = \max_{x,x'}\|P^{s}(x,\cdot) - P^s(x',\cdot)\|_{\rm{TV}}$ is sub-multiplicative, we obtain that 
\[
\bar{d}(\ell t) \leq (1-c_1)^{\ell}.
\]
So, choosing $\ell$ such that $(1-c_1)^{\ell}\leq 1/4$, we get
\[
\max_{x}\|P^{\ell t}(x,\cdot) - \pi\|_{\rm{TV}} \leq \bar{d}(\ell t) \leq \frac{1}{4},
\]
and hence this gives that 
\begin{align}\label{eq:uppermix}
\tmix\lesssim n^{3/2}.
\end{align}
Using~\eqref{eq:lem12} again also gives that for all $x$ and $y$
\[
\left| \frac{P^t(x,y)}{\pi(y)} -1 \right| \leq (c_2-1)\vee (1-c_1),
\]
and hence this implies that $d^\infty(t)\leq (c_2-1)\vee (1-c_1)$. Using~\eqref{eq:uppermix} and Lemma~\ref{lem:linfdist} now proves that $\tunif \lesssim n^{3/2}$.
\end{proof}

Before giving the proof of the lower bound on the mixing time, we state and prove a preliminary lemma on the concentration of a biased random walk on $\Z$ when the laziness parameter is bounded away from $1$.

\begin{lemma}\label{lem:highprob}
Let $\delta\in (0,1]$ and $Y$ be a birth and death chain on $\Z$ starting from $0$ with 
\[
p(x,x)\leq 1-\delta \quad \text{and}\quad p(x,x+1) = \frac{2(1-p(x,x))}{3} = 1-p(x,x) - p(x,x-1).
\]
Then for all $t$ and all positive constants $c<1$, there exists a positive constant $C$ and an interval $I_t = [k-C\sqrt{t}, k+C\sqrt{t}]$ with
\[
\prstart{Y_t\in I_t}{0}\geq c,
\]
where $k$ is such that $t-\estart{\tau_k}{0}>0$ is minimised.
\end{lemma}


\begin{proof}[\bf Proof]
	
	It is elementary to check that there exist two positive constants $c_1$ and $c_2$ so that for all~$k\in \N$ we have 
	\[
	c_1 k \leq \estart{\tau_k}{0} \leq  c_2 k \quad \text{and} \quad \vr{\tau_k} \asymp k.
	\]
	Since $\estart{\tau_k}{0}$ increases by at most $c_2$ when $k$ increases by $1$, for every given $t$ we can find the first~$k$ so that $0\leq \estart{\tau_k}{0}-t\leq c_2$. Let $C$ be a positive constant to be determined.	Since at every time step $Y$ moves by at most $1$, we immediately get that
	\begin{align}\label{eq:yfar}
		\pr{Y_t\notin (k-C\sqrt{t}, k+C\sqrt{t})} \leq \pr{\tau_k>t+C\sqrt{t}} + \pr{\tau_k<t-C\sqrt{t}}.
	\end{align}
	Taking now $C$ sufficiently large, using that $\vr{\tau_k}\asymp k\asymp t$ and Chebyshev's inequality we get
	\begin{align*}
		\pr{\tau_k>t+C\sqrt{t}} &= \pr{\tau_k-\E{\tau_k}> t-\E{\tau_k}+C\sqrt{t}} \leq \pr{\tau_k-\E{\tau_k}>C\sqrt{t}-c_2}\\
		&\leq \frac{\vr{\tau_k}}{(C\sqrt{t}-c_2)^2} \leq \frac{1-c}{2}.
	\end{align*}
Similarly, we obtain that 
\[
\pr{\tau_k<t-C\sqrt{t}}\leq \frac{1-c}{2},
\]
and plugging these two bounds in~\eqref{eq:yfar} completes the proof.
\end{proof}

\begin{proof}[\bf Proof of Theorem~\ref{thm:robustness}](lower bound)

Clearly it suffices to prove the lower bound on $\tmix$.

Let $Y$ be the walk on $G$ as in Definition~\ref{def:walkY}. Let $E=\{X_t=Y_t \text{ or } X_t=\pi(Y_t), \, \forall t\leq n^2\}$. Then Lemma~\ref{lem:coupling} gives us that $\pr{E^c}\leq c e^{-cn}$ for some positive constant $c$.

So we are now going to work with the walk $Y$. Let $t=\epsilon n^{3/2}$ for $\epsilon$ small enough to be determined later. Then by this time, the total number of completed rounds $k(t)$ satisfies $k(t)\leq \epsilon \sqrt{n}$. For each $k\leq \epsilon\sqrt{n}$ among the $k$ rounds, we let $\ell(k)$ and $r(k)$ be the number of left and right rounds completed respectively. Let $(\xi_i)$ be i.i.d.\ Bernoulli random variables taking values $1$ and $0$ depending on whether the $i$-th round completed was on the left or on the right. Then
\[
\ell(k) =\sum_{i=1}^{k}\xi_i \quad \text{and}\quad r(k) = \sum_{i=1}^{k} (1-\xi_i).
\]
Therefore we get
\[
\ell(k) - r(k) = \sum_{i=1}^{k} (2\xi_i -1),
\]
which shows that the difference $\ell(k)-r(k)$ is a simple random walk on $\Z$. Let $F$ be the event 
\[
F=\left\{ \max_{k\leq \epsilon\sqrt{n}} |\ell(k)- r(k)| < \sqrt{C\epsilon} n^{1/4}\right\}. 
\]
Applying Doob's maximal inequality we now get
\begin{align}\label{eq:doob}
\pr{F^c} \leq \frac{\E{(\ell(\epsilon\sqrt{n}) -r(\epsilon\sqrt{n}))^2}}{C\epsilon\sqrt{n}} = \frac{1}{C}.
\end{align}
Recall as before $(T_i^1)_i$ and $(T_i^2)_i$ are independent i.i.d.\ collections of random variables distributed according to the commute time between $0$ and $a$ on the left and $0$ and $b$ on the right cycle respectively. Let $\E{T_1^1} = \beta n$ and $\E{T_1^2}=\gamma n$, where $\beta$ and $\gamma$ are given by Lemma~\ref{lem:lazyalpha} and satisfy $\beta>\gamma$. The exact values of $\beta$ and $\gamma$ will not be relevant here. For all $k\leq \epsilon \sqrt{n}$ we define 
\begin{align*}
M_k = \sum_{i=1}^{k}\xi_i T_i^1 + \sum_{i=1}^{k}(1-\xi_i) T_i^2 - \beta n\sum_{i=1}^{k} \xi_i - \gamma n \sum_{i=1}^{k}(1-\xi_i).
\end{align*}
Then clearly this is a martingale and applying Doob's maximal inequality again we obtain for a large enough constant $C'$
\begin{align}\label{eq:doobagain}
\pr{\max_{k\leq \epsilon\sqrt{n}}|M_k| \geq C' n^{\frac{3}{4}}} \leq \frac{\E{M^2_{\epsilon \sqrt{n}}}}{(C')^2 n^{\frac{3}{2}}} = \frac{1}{C}.
\end{align}
The last equality follows from the fact that $\vr{T_i^1}\asymp n$ for $i=1,2$, and hence $\E{M_k^2} \asymp k n$ for all $k$. Let $A$ be the complement of the event appearing in~\eqref{eq:doobagain}. Recall $k(t)$ denotes the number of rounds completed up to time $t$. Let $B$ be the event
\[
B=\left\{ \left|\sum_{i=1}^{k(t)} \xi_i T_i^1 + \sum_{i=1}^{k(t)}(1-\xi_i) T_{i}^2 -t   \right|< 2C\beta n\right\}.
\]
On $B^c$, the $k(t)+1$ round lasts longer than $2C\beta n$. Markov's inequality then gives
\begin{align}\label{eq:markovb}
\pr{B^c} \leq \pr{T_{k(t)+1}^1\geq 2C\beta n} +\pr{T_{k(t)+1}^2\geq 2C\beta n} \leq \frac{1}{C}.
\end{align}

Let $k_L$ and $k_R$ denote the left and right rounds respectively completed up to time $t$. On~$A$ we have
 \begin{align*}
 \left| \sum_{i=1}^{k(t)}\xi_i T_i^1 + \sum_{i=1}^{k(t)} (1-\xi_i)T_i^2 -(k_R\cdot \beta n + k_L\cdot \gamma n) \right| \leq C n^{\frac{3}{4}},
 \end{align*}
and hence on  $A\cap B$ we have
\begin{align*}
\left|k_R\cdot \beta n + k_L\cdot \gamma n  -t\right|\leq c_1 n,
\end{align*}
where $c_1$ is a positive constant.
If we now fix $|k_R-k_L|=j$, then from the above inequality, we see that the number of choices for $k_R$ and $k_L$ is at most $2c_1$.

Suppose that we know that by time $t$ the random walk on the two cycles performed $k_L$ left rounds and $k_R$ right rounds and that $|k_R-k_L|=j$. Fix an ordering of the rounds and consider a biased $(2/3,1/3)$ random walk $Z$ on $\Z$ for which the laziness parameter is either $\alpha$ or $1/2$ in $[(i-1)n +\frac{n}{4}, (i-1)n + \frac{3n}{4}]$  depending on whether the corresponding round was a left or a right one and equal to $1/2$ everywhere else. Let $\tau_\ell$ be the first time that $Z$ hits $\ell$. Then for every~$\ell$ we can express $\estart{\tau_\ell}{0}$ as follows
\[
\estart{\tau_\ell}{0}=\sum_{i=1}^{m}\estart{\tau_{in}}{(i-1)n} + \estart{\tau_\ell}{m},
\]
where $m$ is the largest integer satisfying $|m-in| \leq n$. Then for all $i$ we have 
\begin{align*}
	\estart{\tau_{in}}{(i-1)n} = \estart{\tau_{in}\1(\tau_{(i-1)n-n/4}<\tau_{in})}{(i-1)n} + \estart{\tau_{in}\1(\tau_{(i-1)n-n/4}>\tau_{in})}{(i-1)n}.
\end{align*}
But $\estart{\tau_{in}\1(\tau_{(i-1)n-n/4}<\tau_{in})}{(i-1)n} \lesssim e^{-c n}$, and hence the difference that we get from reordering the rounds is negligible, once we fix the laziness of the last interval. For each different choice of $j\leq \sqrt{\epsilon} n^{1/4}$ and a choice of $k_R$ and $k_L$ satisfying $|k_R-k_L|=j$ let~$\ell$ be such that $t-\estart{\tau_\ell}{0}>0$ is minimised. Since the ordering of the rounds does not affect $\estart{\tau_\ell}{0}$ up to terms exponentially small in~$n$, for all $\delta>0$ from Lemma~\ref{lem:highprob} we get a positive constant $c_2$ and an interval $I_t=(\ell-c_2\sqrt{t}, \ell+ c_2\sqrt{t})$ so that 
\begin{align}\label{eq:probyhigh}
\pr{Z_t\in I_t}\geq 1-\delta.
\end{align}
As we proved above on the event $A\cap B\cap G$ there are $M=c_2 \sqrt{\epsilon} n^{1/4}$ different choices for the number of left and right rounds and each of them gives rise to an interval $I_t$. Let $I$ be the union of all these~$M$ intervals. Then $|I| \lesssim \sqrt{\epsilon} n^{1/4} \sqrt{t} \asymp \epsilon n$. If we now set $\til{I}=I\bmod n$ and take two copies, one on each cycle, then from~\eqref{eq:doob}, \eqref{eq:doobagain} and~\eqref{eq:markovb} we get for $n$ sufficiently large
\begin{align*}
	\pr{X_t\notin \til{I}} \leq \pr{A^c} + \pr{B^c} +\pr{F^c} + \pr{E^c}+ \pr{A, B, F, Y_t\notin \til{I}} \leq \frac{4}{C} + \pr{A, B, F, Y_t\notin \til{I}}.
\end{align*}
We now have
\begin{align*}
	\pr{A, B, F, Y_t\notin \til{I}} \leq  \sum_{j\leq \sqrt{\epsilon}n^{1/4}}  \sum_{\substack{i_1,\ldots, i_{\epsilon\sqrt{n}}\in \{0,1\}\\ |\#\{i=1\}-\#\{i=0\}|=j}} \pr{Z^i_t\notin I}\pr{\xi_1=i_1,\ldots, \xi_{\epsilon\sqrt{n}}=i_{\epsilon\sqrt{n}}} + e^{-cn},
\end{align*}
where $Z^i$ stands for a walk on $\Z$ with the laziness parameter arranged according to the vector $i=(i_1,\ldots, i_{\epsilon\sqrt{n}})$. Note that to get the inequality above, once we fix the left and right rounds, we coupled in the obvious way the walk $Y$ to the walk $Z\bmod n$ and the probability that the coupling fails is $e^{-cn}$. Using~\eqref{eq:probyhigh} now gives that 
\[
\pr{A, B, G, Y_t\notin \til{I}} \leq  \delta + e^{-cn}.
\]
Putting all estimates together shows that when $t=\epsilon n^{3/2}$ and $I$ and $\til{I}$ are as above, for a sufficiently large constant~$C$ we obtain
\[
\pr{X_t\in \til{I}} \geq 1-2\delta.
\]
Since $\pi(\til{I}) \leq c_3\epsilon$, for a positive constant $c_3$, it immediately follows that 
\[
\| P^t(0,\cdot) - \pi\|_{\rm{TV}} \geq 1-2\delta- c_3\epsilon.
\]
Taking $\epsilon$ and $\delta$ sufficiently small finishes the proof of the lower bound.
\end{proof}

\section{Exploration time}\label{sec:exploration}

In this section we prove Theorems~\ref{sec} and~\ref{primo}. We start by introducing notation that will be used throughout the section. For any two vertices $u$ and $v$ in $G$ we write $H(u,v)=\estart{\tau_v}{u}$ and if $A$ is a set we write $H(u,A) = \estart{\tau_A}{u}$. We write $\mathcal{C}(u,v)$ for the commute time between $u$ and $v$, i.e.\ $\mathcal{C}(u,v)= H(u,v)+H(v,u)$. Finally, if $e$ is an edge, then we write $\mathcal{C}(e)$ for the commute time between the endpoints of $e$. 

We next note the following two inequalities for commute and cover times in Eulerian digraphs.

\begin{lemma}\label{com}
Consider two vertices $v$ and $w$ such that ${d}(v,w)= \ell$ (in the undirected version of~$G$), then the commute time between~$v$ and~$w$ is bounded by $m \ell$.
\end{lemma}

\begin{proof}[\bf Proof]
Let us first assume $\ell=1$, and write $d$ for $d_{out}(v)$. On each excursion from the vertex $v$ to itself, $w$ is hit with probability at least $\frac{1}{d}$, since $w$ can be hit on the first step of the excursion. So the expected number of excursions needed before visiting $w$ is at most $d$. Each excursion has an expected length of $\frac{m}{d}$, so by Wald's identity, the expected number of steps before hitting $w$ and going back to $v$ is at most $\frac{dm}{d}= m$.

In general, let $v=v_0 \rightarrow v_1 \rightarrow \ldots \rightarrow v_\ell =w$ be a path from $v$ to $w$ in the undirected version of $G$. Then we may bound 
\begin{equation}\label{crude}
\mathcal{C}(u,v) \leq \sum_{i=0}^{\ell-1} \mathcal{C} (v_i,v_{i+1}).
\end{equation}
We just saw each summand on the right hand side is less than $m$, so this finishes the proof.
\end{proof}

\begin{claim}\label{cl:coveul}
There exists a positive constant $c$ so that the following is true. Let $G$ be an Eulerian digraph on $n$ vertices, $m$ directed edges and minimal degree $d_{\min}$. Then for all $v\in V$ we have 
\[
\estart{T_n}{v} \leq 16\cdot \frac{mn}{d_{\min}}. 
\]
\end{claim}

\begin{proof}[\bf Proof]

In~\cite[Theorem~2]{KSNL} it was shown that in every undirected graph $G$ there is a spanning tree~$T$ so that for all $v$
\[
\estart{T_n}{v}\leq \sum_{e\in T} \mathcal{C}(e) \leq 16 \cdot \frac{mn}{d_{\min}}. 
\]
It follows immediately from the min-max characterisation of commute times in~\cite{Doyle} that $\mathcal{C}(e)$ in an Eulerian digraph is bounded above by the commute time $\til{\mathcal{C}}(e)$ in the corresponding undirected graph and this completes the proof.
\end{proof}

\subsection{Proof of Theorem~\ref{sec}}

We start by generalizing some well-known facts from the reversible case to our more general setting. Note that we are not optimising the constants in the results below. 

The first lemma is a well known graph-theoretic fact, but we include the proof for the reader's convenience.

\begin{lemma}\label{soares}
Consider an undirected graph $G$ on $n$ vertices and assume each node has at least $d$ distinct neighbours. Let $A \subseteq V$ be a set of cardinality $|A| < n$. Then,
\begin{equation}
d(v,A^{c}) \leq \frac{3 |A|}{d} + 1.
\end{equation}
\end{lemma}

\begin{proof}[\bf Proof]

Let $w \in A^c$ be a vertex closest to $v$ among all vertices in $A^c$. Let $v = v_0,v_1,\ldots,v_\ell = w$ be the shortest path leading from $v$ to $ w$ with $\ell = d(v, A^c)$.  If $\ell = 1$ the statement holds trivially, so we may assume $\ell \geq 2$. Note that $v_0,\ldots, v_{\ell - 2}$ have all their neighbours in $A$, since otherwise this would contradict the minimality of $\ell$. Moreover, note that $(v_{3i})_{0 \leq 3i \leq \ell - 2}$ must have disjoint neighbours, as otherwise this would also contradict the minimality of $\ell$. 
There are $1 + \lfloor(\ell -2)/3 \rfloor$ indices $i\geq 0$, such that $3i \leq \ell -2$. Hence we find that 
\begin{equation}
d \left(1+ \left\lfloor \frac{\ell - 2}{3} \right\rfloor \right) \leq \lvert A \rvert,
\end{equation}
which in turn implies $(\ell- 1)/3 \leq \frac{\lvert A\rvert }{d}$.
\end{proof}

%
%

Throughout this section we make the assumption that $G$ is $d$-regular.

Define $N_v(t) := \# \left\{s < t : X_s =v\right\}$ the number of visits to $v$ up to time $t$ and for a subset $A \subseteq V$, define $N_v(A) := N_v(\tau_A)$ the number of visits to $v$ before reaching the set~$A$.  A version of the following two lemmas is given in \cite[Chapter~6]{AF}, where everything is done in the undirected setting. We adapt the proofs to our broader context.

\begin{lemma}\label{return}
Let $A \subseteq V$ and $v\in A$. Then,
\begin{enumerate}
\item[{\rm{(1)}}] $\estart{\tau_{A^c}}{v} \leq 10|A|^2$ \label{argo}
\item[{\rm{(2)}}] $\estart{N_v(A^c)}{v} \leq 10|A|$
\item[{\rm{(3)}}] $\estart{N_v(t)}{v} \leq 8\sqrt{t}$ for all $t \leq 10n^2$.
\end{enumerate}
\end{lemma}
\begin{proof}[\bf Proof]
{\rm (1)} We can assume that $d \leq 2|A|$, otherwise the result is trivial.
Contract $A^c$ to a single vertex. 
The resulting graph has at most $2d|A|$ directed edges (excluding self-loops) and is still Eulerian. Lemma~\ref{soares} tells us there is a path of length $\ell \leq \frac{3 |A|}{d}+1$ from $v$ to $A^c$ in the undirected version of $G$. So Lemma \ref{com} now gives that the expected time to hit $A^c$ is smaller than the commute time which is in turn upper bounded by $2d|A|\cdot \left(\frac{3 |A|}{d}+1\right)\leq 10|A|^2$.

{\rm{(2)}} Let $\tau_v^{A^c}$ be the first time that the random walk returns to~$v$ after having first hit $A^c$. Note that the proof of part (1) gives that $\estart{\tau_v^{A^c}}{v}\leq 10|A|^2$. Then~using~\cite[Lemma~10.5]{LevPerWil} we get that 
\[
\estart{N_v(A^c)}{v} = \estart{\tau_{v}^{A^c}}{v} \pi(v) \leq 10|A|^2\cdot \frac{1}{|V|} \leq 10|A|,
\]
since $A\subseteq V$.

{\rm{(3)}} For $s >0$ to be chosen later let
\begin{equation*}
A = \{ w : \estart{N_v(t)}{w} > s \}.
\end{equation*}
Since $G$ is assumed to be Eulerian, $P$ is bistochastic and, for all $\ell \in \mathbb{N}$,
\begin{equation*}
\sum_w p_{wv}^{(\ell)} = 1,
\end{equation*}
which in turn implies, summing over $\ell\leq t-1$, that $$\sum_w \estart{N_v(t)}{w} = t,$$ so $|A| \leq t/s$. As long as $\frac{t}{s} < n$, we can now use the result above to get
\begin{equation*}
\estart{N_v(t)}{v} \leq \estart{N_v(A^c)}{v}+s \leq 10 \frac{t}{s} +s.
\end{equation*}
Choosing $s = \sqrt{10t}$ restricts $t$ to be smaller than $10n^2$, and gives the desired bound.
\end{proof}

\begin{proof}[\bf Proof of Theorem~\ref{sec}]
Let $\alpha$ and $C$ be constants to be chosen later to satisfy $\alpha^2 C\leq 10$. We can always assume $k \leq \alpha n$, otherwise we can bound $T_k$ by the cover time, which is smaller than $16n^2 \leq 16 k^2/\alpha^2$, using Claim~\ref{cl:coveul}. Let $t = Ck^2$. 
Observe that if we knew that for all $v$
\begin{equation}\label{mille}
\prstart{T_k \geq  t}{v} \leq \frac{1}{2},
\end{equation}
then the strong Markov property would imply that $T_k$ is stochastically dominated by $t \times \xi$ where $\xi \thicksim \text{Geom}(\frac{1}{2})$, which is enough to conclude.

Let $v_1, v_2, \ldots$ be the distinct random vertices visited by the walk, in the chronological order they appeared. Write $t_i$ for the hitting time of the vertex~$v_i$. Conditioning on $v_i,t_i$ and using the strong Markov property we get
\begin{equation} \label{ezek}
\estart{N_{v_i}(t)}{v} = \sum_{u,s} \mathbb{P}_v(v_i =u ,t_i = s)\estart{N_u(t-s)}{u} \leq 8 \sqrt{t},
\end{equation}
where in the last step we used Lemma~\ref{return}, since $t\leq 10n^2$ by the choice of $\alpha$ and $C$.
We use the convention that $N_u(\cdot)$ is $0$ when evaluated at negative times. 
Markov's inequality together with~\eqref{ezek} now gives 
\begin{align*}
\prstart{T_k\geq  t}{v} = \prstart{\sum_{i=1}^{k-1} N_{v_i}(t)=t}{v} \leq \frac{8k\sqrt{t}}{t} = \frac{8}{\sqrt{C}}, 
\end{align*}
where in the last equality follows from the choice of $t$. Taking $\sqrt{C}=16$ and $\alpha$ such that $\alpha^2 C=10$ proves~\eqref{mille} and this finally shows that 
\begin{equation*}
\estart{T_k}{v} \leq  \max \left(2C, \frac{16}{\alpha ^2}\right) k^2 =512 k^2,
\end{equation*}
which concludes the proof.
\end{proof}

\subsection{Proof of Theorem \ref{primo}.}
For an undirected graph $G = (V,E)$ we write $G^k$ for the graph $(V,\widetilde{E})$ where $(u,v) \in \widetilde{E}$ if and only if there is a path from $u$ to $v$ of length at most~$k$ in G. 
Recall a Hamiltonian cycle is a cycle in the graph that visits all vertices exactly once. A graph is called Hamiltonian if it contains a Hamiltonian cycle.

As in \cite{BF}, the following fact will be crucial.  We include the proof for completeness. 
\begin{lemma}\label{tour}
For any undirected graph $G$, its cube $G^3$ is Hamiltonian.
\end{lemma}

\begin{proof}[\bf Proof]

By taking a spanning tree of~$G$ we can reduce to the case where $G$ is a tree.
We will prove the following stronger statement by induction : for any $v \in G$, there exists a labelling of the vertices $v=v_1,\ldots,v_n$ of $G$ such that 
\begin{itemize}
\item $d(v_i,v_{i+1}) \leq 3$ for all $i \leq n-1$ and 
\item $d(v_1,v_n) = 1$. 
\end{itemize} 

We will call a labelling proper if it fulfils the previous two conditions.

The result obviously holds when the tree has size $1$. Assume it holds for trees of size $n-1$ and let the size of $G$ be~$n$. Let $v \in G$ and~$w$ a neighbour of~$v$. Since $G$ is a tree, $G - (v,w)$ has two connected components $G_v$ and $G_w$, containing $v$ and $w$ respectively.

By the induction hypothesis we can properly label the vertices of $G_v$ as $v = v_1,\ldots,v_{\ell}$ and properly label the vertices of $G_w$ as $w=w_1,\ldots,w_{r}$, where $r$ and $l$ are related by $n = \ell + r$. 

We build the labelling on $G$ as follows (running forward on $G_v$ and backwards on $G_w$).
\begin{equation*}
\widetilde{v}_i = \begin{cases} v_{i} &\mbox{if } 1 \leq i \leq \ell \\ 
w_{n+1-i} & \mbox{if } \ell +1 \leq i \leq n \end{cases}
\end{equation*}
Now observe that
\begin{equation*}
d(\widetilde{v}_{\ell},\widetilde{v}_{\ell+1}) = d(v_{\ell},w_r) \leq d(v_1,v_{\ell}) + d(v_1,w_1) + d(w_1,w_r)=3,
\end{equation*}
since the labellings on $G_v$ and $G_w$ are proper and $w_1=w$ and $v_1=v$ are neighbours. Finally, note that $d(\widetilde{v}_n,v_1) = d(w,v) =1$.
\end{proof}

Before proving Theorem~\ref{primo} we define phases for the walk and the notions of good and bad vertices following closely~\cite{BF}.

Call $v$ the starting vertex. Let $v_1, \ldots , v_n$ be the arrangement (i.e.\ a cycle in the graph $G^3$) we get using Lemma~\ref{tour} on the undirected version of~$G$.

A naive attempt would be to start a new phase each time a new vertex is discovered, but this seems hard to analyse. Instead, we will start a new phase each time the walk moves either to a set of \emph{good vertices} (to be defined) or when it reaches the next vertex specified by the ordering of the cycle. In the latter case, we say the walker \emph{moves along the cycle}. We emphasise that the orientation on the cycle has nothing to do with the orientation of ${G}$, i.e.\ we know nothing about ${d}(v_i,v_{i+1})$, since the cycle is built ignoring the directions on the edges of ${G}$.

We will show $O(k)$ phases are needed and they all have $O(k^2)$ expected length. This will be enough to show Theorem~\ref{primo}.

The definitions we give below are the same as in \cite{BF}. At the beginning of phase $i$, let $Y_i$ be the visited vertices so far. The last vertex visited in phase $i-1$ will be denoted $s_i$ and $r_i$ is its right successor on the cycle. A vertex $v_j\notin Y_i$ is called good in phase~$i$ if for all $\ell\leq n$,
\begin{equation*}
| \{ v_j,v_{j+1},\ldots,v_{j+\ell-1} \} \cap Y_i | \leq \frac{\ell}{2}.
\end{equation*}
Note that if $k>n$, then we take $v_k$ to be $v_{k-n}$, i.e.\ we always move around the cycle.

We define $U_i$ to be the set of all good vertices in phase~$i$. The set of vertices $B_i=V - (U_i \cup Y_i)$ are called bad (see Figure~\ref{cyc1}). 
\begin{figure}
\center{
\includegraphics[scale=0.4]{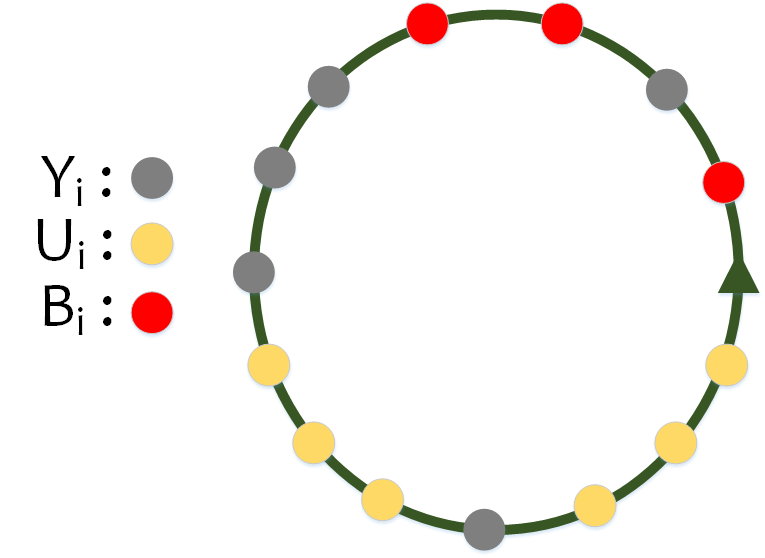}
\caption{The different subsets of the vertices.}
}
\label{cyc1}
\end{figure}

\begin{lemma}[\cite{BF}]\label{construction}
At most  $2k$ phases are needed before visiting~$k$~distinct vertices.
\end{lemma}
\begin{proof}[\bf Proof] 

The proof can be found in~\cite{BF}, but we briefly sketch it here. 
Recall phases end either when the walk moves to the right on the cycle or when it moves to a good vertex. By definition, for any length $\ell \leq n$, at least half of the vertices following a good vertex on the cycle are unvisited, so as long as the walk keeps moving on the cycle, half of the starting vertices of the phases are new. 
\end{proof}

At this point, we are still working on the undirected version of ${G}$, and the following combinatorial fact still holds.
\begin{lemma}[\cite{BF}]\label{badgood}
If $|Y_i| \leq \frac{n}{2}$, then $|B_i| \leq |Y_i|$.
\end{lemma}
\begin{proof}[\bf Proof]
We again proceed as in \cite{BF}. The proof is by induction. If $|Y_i|=1$, then $B_i$ is empty and the result holds. Suppose it holds for $|Y_i| \leq \ell$, with any $n$ such that $\ell < \frac{n}{2}$ and let $Y_i$ be of size~$\ell +1$.

As $\ell+1 \leq \frac{n}{2}$, we can find $v \in Y_i$ that is the right neighbour (on the cycle) of some vertex in $w \in Y_i^c$. Remove both $v$ and $w$, and apend $w$'s left neighbour to $v$'s right neighbour. In this new graph of size $n-2$, (or we should say cycle, since all that matters is the cycle) $|\widetilde{Y}_i|=\ell-1 \leq \frac{n-2}{2}$, and so by induction hypothesis, $|\widetilde{B}_i| \leq |\widetilde{Y}_i|$. But now notice $Y_i = \widetilde{Y}_i \cup \{v \} $, and $B_i \subseteq \widetilde{B}_i \cup \{w\}$, so we can conclude.
\end{proof}

So far we have been following~\cite{BF}. The proof there for the undirected case is completed by noting that if $s\in W\cap \partial W^c$, then $H(s,W^c) \leq 2m(W)$, where we recall that $H(a,B)$ stands for the expected hitting time of $B$ starting from $a$ and $m(W)$ is the number of edges of $W$. Using this, the authors in~\cite{BF} bound the expected length of a phase and together with Lemma~\ref{construction} this concludes the proof. 

However, for the Eulerian directed case obtaining an upper bound on the expected length of a phase requires extra care. This is because of the fact that an arbitrary subset of the vertices doesn't span an Eulerian digraph, and all we know is that either $(s_i,r_i) \in {E}$ or $(r_i,s_i) \in {E}$, since the cyclic arrangement was chosen for the undirected version of ${G}$. The following lemma solves this issue. Note that we write $A\sqcup B$ to mean the disjoint union of $A$ and $B$. Also we denote by $\partial A$ the set of undirected neighbours of $A$ and we write $E(A,B)$ for the set of directed edges from $A$ to $B$.

\begin{lemma}\label{key}
Let $W \sqcup Z = V$ and $s \in W \cap \partial Z$. Then $H(s,Z) \leq 12|W| ^2$.
\end{lemma}

\begin{proof}[\bf Proof]

We first claim that if $A\sqcup B= V$ and $a\in A\cap \partial B$, then 
\begin{align}\label{eq:easy}
H(a,B)\leq |A|^2+ 2|{E}(A,B)|.
\end{align}
Indeed, the graph obtained by contracting $B$ to a single vertex is still Eulerian and has at most~$|A|^2+2|{E}(A,B)|$ edges. Some edges to or from~$B$ might be multiple edges. Nevertheless, Lemma \ref{com} covers this case and gives $H(a,B) \leq C(a,B) \leq |A|^2 + 2|{E}(A,B)|$.

We now define $W_1 =\{ v \in W : \text{deg}(v) \leq 2|W| \}$ and $W_2 = W\setminus W_1$. 
It then follows that 
\begin{align}\label{eq:equal}
W_1\cap \partial W_2\subseteq W_1\cap \partial (W_2\sqcup Z).
\end{align}
Next we let 
\begin{align*}
t_1 &:= \max \{H(s, Z) : s \in W_1\cap \partial(W_2\sqcup Z)\},  \\
t_2 &:= \max \{H(s, Z) : s \in W_2 \}.
\end{align*}
Let $s\in W_1\cap \partial (W_2\sqcup Z)$. Then 
\begin{align*}
H(s,Z) \leq H(s,W_2\sqcup Z) + t_2.
\end{align*}
Since $s\in  \partial (Z\sqcup W_2)$, using~\eqref{eq:easy} we get
\[
H(s,W_2\sqcup Z) \leq |W_1|^2+ 2 \cdot (2|W|\cdot |W_1|) \leq 5|W|^2,
\]
where the bound on $|{E}(W_1,W_2\sqcup Z)|$ follows from the definition of $W_1$. Putting the last two inequalities together gives
\begin{align}\label{eq:t1}
t_1\leq 5|W|^2 + t_2.
\end{align}
Let $s^* \in W_2$ be such that, $H(s^*,Z) = t_2$. Conditioning on the first step of the walk yields
\begin{equation}\label{eq:t2}
t_2 \leq \frac{1}{2} + \frac{1}{2}(1+ \alpha t_2 + (1-\alpha)t_1).
\end{equation} 
Indeed, when the walk is in $W_2$ it exits to $Z$ with probability at least half by construction. Conditioning on the walk not exiting to $Z$, the parameter $\alpha \in [0,1]$ is the probability that the walk started at $s^*$ stays in $W_2$ after one step. Under the same conditioning, with probability $1-\alpha$, the walk will move to some vertex in $W_1\cap \partial W_2$ which together with~\eqref{eq:equal} explains the last term on the right hand side.

Plugging~\eqref{eq:t1} in~\eqref{eq:t2}, we get $t_2 \leq 2+5|W|^2$, and hence $t_1 \leq 2+10|W|^2\leq 12|W|^2$.
Therefore, we showed that if $s\in W\cap \partial Z$, then $H(s,Z)\leq 12|W|^2$ in all cases and this finishes the proof of the lemma.
\end{proof}

\begin{proof}[\bf Proof of Theorem~\ref{primo}]

If $2k > n$, we bound $\E{T_k}$ by the expected cover time which is smaller than $n^3 \leq 8k^3$ by Claim~\ref{cl:coveul}. So we can assume that $2k\leq n$.

Let $\Phi_i$ be the length of phase $i$ before having visited $k$ distinct vertices. Then using Lemma~\ref{key} with $W=Y_i\cup B_i\setminus\{r_i\}$, $Z= \{r_i\}\cup U_i$ and $s=s_i$ the starting vertex of phase~$i$, we obtain 
\[
\E{\Phi_i} = \E{\econd{\Phi_i}{Y_i, s_i}} \leq \E{3\times 12(2|Y_i|)^2} = 144 \E{|Y_i|^2},
\]
where the factor~$3$ comes from the fact that~$r_i$ is not necessarily a neighbour of~$s_i$, but~$d(s_i,r_i)\leq 3$. We also have
\begin{align}\label{eq:412}
	\E{\Phi_i \1(Y_i \leq k)} \leq 144 k^2.
\end{align}
Since the number of phases before discovering $k$ vertices is not greater than $2k$, we can now write
\begin{align*}
	T_k = \sum_{i=0}^{\infty} \Phi_i \1(Y_i \leq k) =  \sum_{i=0}^{2k} \Phi_i \1(Y_i \leq k).
\end{align*}
Taking expectations, and plugging~\eqref{eq:412} yields
\begin{align*}
	\E{T_k} =  \sum_{i=0}^{2k} \E{\Phi_i \1(Y_i \leq k)} \leq 2k \times 144 k^2 = 288 k^3
\end{align*}
and this concludes the proof.
\end{proof}

\section*{Acknowledgements}
The authors would like to thank Omer Angel, Robin Pemantle and Mohit Singh  for useful discussions. The first author is also grateful to Laurent Massouli\'e for giving him the opportunity to spend time at Microsoft Research in Redmond.

\bibliographystyle{plain}
\bibliography{biblio}

\end{document}